\documentclass[12pt, final]{article}

\usepackage[utf8x]{inputenc}
\usepackage[T1]{fontenc}

\fontsize{12pt}{18pt}\selectfont

\usepackage{amsmath,amsfonts,amssymb,amsthm}
\usepackage{bbm}
\usepackage{graphicx}
\usepackage[notref, notcite]{showkeys}
\usepackage{mathtools}
\usepackage{tabularx}
\usepackage[hyperfootnotes=false]{hyperref}
\usepackage{color}
\usepackage[a4paper, textheight=21cm, textwidth=16cm]{geometry}

\usepackage{xargs}                      
\usepackage[pdftex,dvipsnames]{xcolor}
\usepackage[colorinlistoftodos,prependcaption,textsize=footnotesize]{todonotes}

\usepackage{mathtools}
\DeclarePairedDelimiter{\ceil}{\lceil}{\rceil}

\newcommand{\R}{\mathbb{R}}

\newcommand{\rs}{r^{\ast}}

\newtheorem{prop}{Proposition}
\newtheorem{theorem}{Theorem}
\newtheorem{lemma}{Lemma}
\newtheorem{remark}{Remark}
\newtheorem{definition}{Definition}
\newtheorem{cor}{Corollary}

\pagecolor{white}

\newcommand\blfootnote[1]{%
  \begingroup
  \renewcommand\thefootnote{}\footnote{#1}%
  \addtocounter{footnote}{-1}%
  \endgroup
}

\begin{document}

\phantom{ }
\vspace{4em}


\begin{flushleft}
{\large \bf Lower bounds on mixing norms for the advection diffusion equation in $\R^d$}\\[2em]
{\normalsize \bf Camilla Nobili$^1$ and Steffen Pottel$^2$}\blfootnote{\today}\\[0.5em]
\small \begin{tabbing} 
$^1$ \= Department of Mathematics, University of Hamburg, 20146 Hamburg, Germany.\\
\>E-mail: camilla.nobili@uni-hamburg.de\\
$^2$ \> K\"uhne Logistics University, 20457 Hamburg, Germany.\\
\>E-mail: steffen.pottel@the-klu.org\\
\end{tabbing}
\end{flushleft}

{\small
{\bf Abstract:} An algebraic lower bound on the energy decay for solutions of the advection-diffusion equation in $\R^d$ with $d=2,3$ is derived using the Fourier-splitting method.
Motivated by a conjecture on mixing of passive scalars in fluids, a lower bound on the $L^2$- norm of the inverse gradient of the solution is obtained via gradient estimates and interpolation.
}


\vspace{2em}

\section{Introduction}

We are interested in the long-time behavior of a diffusive concentration field $\theta$ in $\R^d$ with $d=2,3$, which is advected by a time-dependent divergence-free vector field $u$.
The dynamics of $\theta$ are described by the advection-diffusion equation
\begin{equation}\label{AD}
\begin{cases}
 \partial_t\theta+u\cdot \nabla \theta-\kappa\Delta \theta=0 & \mbox{ in } \R^d\times (0,\infty)\\
 \nabla \cdot u=0 & \mbox{ in } \R^d\times (0,\infty)\\
 \theta(x,0)=\theta_0(x) & \mbox{ in } \R^d\,,
\end{cases}
\end{equation}
where $\kappa$ is the molecular diffusion coefficient (usually $\kappa \ll 1$).
Long-time asymptotics for this (or a variation of this) equation have been studied in various contexts. 
Among all the results available in the literature, we want to mention the works of Zuazua et.\ al.\ \cite{DURO2000275, ESCO-ZUA}, in which the large time behavior of solutions of diffusion equations with (non-linear) advection term of the type $a\cdot\nabla(|\theta|^{q-1}\theta)$ is studied exploiting the scaling properties of the equation.
For the advection-diffusion equation \eqref{AD} in $\R^2$ with stationary, periodic or random vector fields, Fannjiang and Papanicolaou in \cite{Fannjiang94convectionenhanced, PAPA-random} studied the \textit{effective diffusivity}, defined as the long-time and space average of $|x|^2\theta(x,t)$,  with $\theta_0$ being a Dirac delta function at the origin. 
In \cite{Fannjiang94convectionenhanced}, a lower bound for the effective diffusivity is derived formulating the problem as a variational principle.
While upper bounds for the $L^p$-norms of the solution of \eqref{AD} have been produced under various assumptions on $u$ by means of variational techniques, scaling analysis, and regularity theory (cf.\cite{CarLos}), lower bounds are more subtle and difficult to find in general (see, for example, the results and discussions in \cite{Nobili-Otto} and references therein).
In this direction, we mention the remarkable result of Maekawa in \cite{Maek08}, where a lower bound for the kernel of an advection-diffusion equation was produced under the constraint $\sup_{t>0}t^{\frac 12}\|u(t)\|_{\infty}<\infty$.
Bounds on the $L^2$-norm of $\theta$ further allow to study the problem of mixing, i.e. the evolution of the concentration field in a solvent towards a uniform distribution.
Discussions on various measures for mixing can be found in \cite{Thi2012} and \cite{DN2017}.
The quantity 
  \begin{equation}\label{def-H-1}
  \|\nabla^{-1}\theta(t)\|_2^2 =\int_{\R^d} (|\xi|^{-1} |\hat{\theta}(\xi,t)|)^2\, d\xi\,
  \end{equation} 
is particularly suitable to describe mixing degrees as ``it downplays the role of small scales''\cite{Thi2012} by suppressing small-scale variations.
In this regard, Miles \& Doering \cite{DM2018} extend the consideration to the quantity
 \begin{equation}\label{def-lambda}
  \lambda(t):=\frac{\|\nabla^{-1}\theta(t)\|_2}{\|\theta(t)\|_2} \, ,
 \end{equation}
 called the \textit{filamentation length}.
Their numerical experiments on the torus entail the following interesting fact: $\lambda$ approaches a minimal value for large times, i.e. $\lim_{t\rightarrow \infty}\lambda(t)=\mbox{constant}$, which is the minimal length scale for filaments in presence of diffusion, named Batchelor-scale.
Motivated by these recent results, we are primarily interested in deriving a lower bound for the energy of the solution of equation \eqref{AD}, under constraints on the energy of $u$.
In particular, this bound is the key ingredient for estimating $\|\nabla^{-1}\theta\|_2$ and subsequently deriving  some understanding of the filamentation length $\lambda$. 

In this work, the methods to find the asymptotic behavior of the solution of \eqref{AD} are inspired by the seminal works of Maria Schonbeck on the Navier-Stokes equation \cite{Scho1986, Scho1991} and depend crucially on the decay properties of the vector field and its gradient in time. 
In order to derive a lower bound for $\|\theta(t)\|_2$, we view advection as a forcing term for the heat equation, i.e.
$$\partial_t\theta-\kappa\Delta \theta=-u\cdot \nabla \theta\,.$$
This point of view is very convenient, as it allows us to exploit the representation formula of the solution
$$\theta(x,t)=\int_{\R^d}G(x-y,t) \, \theta_0(y)\, dy+\int_0^t\int_{\R^d} G(x-y, t-s) \, (u\cdot \nabla \theta)(y,s)\, dy\, ds\,,$$
with $G$ being the heat kernel.
Decomposing the scalar field $\theta$ as
$$\theta=T+(\theta-T)\,,$$
where $T$ solves the heat equation, a lower bound on the solution of the advection-diffusion equation follows directly from the combination of a lower bound on the solution of the heat equation $T$ and a suitable upper bound for the difference of solutions $\theta - T$ since for all time the bound
$$\|\theta(t)\|_2\geq \|T(t)\|_2-\|(\theta-T)(t)\|_2 \, ,$$
holds. 
We remark that suitable refers to vector fields $u$ such that $\|T(t)\|_2\geq\|(\theta-T)(t)\|_2$ for large times.
Furthermore detecting necessary assumptions on the data of the present problem, which render the previous inequalities valid, requires a careful analysis.

As mentioned in \cite{BioScho09} the energy decay rate is dependent
on the actual form of the data and not on initial energy. In particular they show that the solution of the heat equation decays at most exponentially if and only if $\theta_0$ is zero in some neighborhood of the origin in Fourier space.
Lower bounds instead can be deduced if the Fourier transform of the initial data is larger than a positive constant in a ball of radius $\delta$ centered at the origin.
Together with the assumption $\theta_0\in L^1\cap L^2$, this class includes Gaussian-like initial data but excludes mean-free initial data, i.e. $\theta_0$ such that
$$0= \hat{\theta}_0(0)=\int \theta_0(x)\,dx\,.$$

This restrictive condition may be relaxed. 
Following the ideas in \cite{BioScho09,niche2015decay,niche2016decay}, we will introduce the notion of \textit{decay character}
\begin{equation}\label{eq:introRS}
\rs=\rs(\theta_0)=\sup\left\{r \in \left(-\frac d2, \infty\right) \Big|\lim_{\delta\rightarrow 0}\delta^{-2r-d}\int_{|\xi|\leq \delta}|\hat{\theta}_0|\, d\xi=0\right\}\,,
\end{equation}
 to describe the decay of the $L^2$ of the initial data at the origin in Fourier space (see Definition \ref{decay-char}).
The number $\rs$ (if it exists) will then play a crucial role in the decay of the solution of the equation.
In Remark \ref{re:BrandoleseRS}, we discuss a result of Brandolese \cite{B16}, which relaxes the requirements on decay characters allowing classes of initial data, for which the $\rs$ of \eqref{eq:introRS} is not well-defined.

Because of the perturbation approach we are using ($\theta$ appears on the right-hand side of the representation formula) the upper bound on the difference $\theta-T$ relies on an upper bound on the $L^2$-norm of the solution $\theta$. 
Using the {\it Fourier-splitting} method we establish that for any divergence-free vector field such that $\|u(t)\|_{L^2}\sim (1+t)^{-\alpha}$ with $\alpha>\frac 12-\frac d4$ we have 
\begin{equation}\label{upper-bound-comp}
   \|\theta(t)\|_2\lesssim C \kappa^{-\max\{\frac d4+\frac{\rs}{2},m\}}(1+t)^{ -\min\{\frac d4+\frac{\rs}{2}, \frac d4+\frac 12\}} \quad \mbox{ for } d=2,3\, ,
\end{equation}
where $m$ is a rational number. This type of result is not new, as similar estimates were proven (for instance) in \cite{niche2015decay,niche2016decay} for dissipative quasi-geostrophic equation, the compressible $3d$ Navier Stokes equations and the Navier-Stokes-Voigt equation.
Dealing with a passive scalar equation, in our analysis the condition on $u$ comes out of the analysis and we can enlarge the class of ``admissible'' velocity field considered in \cite{CarLos}.
While the velocity field in the latter must satisfy $\|u(t)\|_{\infty}\sim t^{-\frac 12}$, in our case any decay of the energy of $u$ is sufficient in $2d$, whereas in $3d$ it might even increase.

The \textit{Fourier-splitting} technique was introduced by Maria Schonbeck in \cite{Schon1985} in order to derive $L^2$-decay estimates for weak solutions to the Navier-Stokes equations. This method, applied to \eqref{AD}, relies on the following observation: the standard energy identity can be written in Fourier space as
$$\frac{d}{dt}\int_{\R^d}|\hat{\theta}(\xi,t)|^2\, d\xi=-2\kappa\int_{\R^d}|\xi|^2|\hat{\theta}(\xi,t)|^2\, d\xi\,,$$
where $\hat{\theta}$ is the Fourier transform of the solution.
With this the term "Fourier-splitting" refers to a decomposition of the frequency domain into two time-dependent subdomains, yielding a first-order differential inequality for the spatial $L^2$-norm of $\hat \theta$. Incidentally, an upper bound on the $L^2$-norm of the solution is produced by estimating the integral of $|\hat \theta(\xi)|^2$ over an $d$-dimensional
sphere centered at the origin with an appropriate time-dependent radius. 
Apart from the previously mentioned works on the Navier-Stokes equation the {\it Fourier splitting} method has been successfully applied to produce upper bounds on the solutions of the Boussinesq systems \cite{BS2012}, the quasi-geostrophic equation \cite{SchoScho2006}, the modified quasi-geostrophic equation \cite{FNP16}, the Camassa-Holm equations \cite{T18}, the electron inertial Hall-MHD system \cite{FZ19}, and for a magneto-micropolar system \cite{NP20}. 
\begin{theorem}\label{th1}
Let  $\theta_0=\theta_0(x)$ satisfy
\begin{equation}\label{assumption-initial-data}
    \theta_0\in L^2(\R^d) \qquad \mbox{with decay character } \rs(\theta_0)=\rs \mbox{ with } -\frac d2<\rs<1.
\end{equation}
Moreover assume  
 \begin{equation}\label{assumption-u}
 \|u(t)\|_{2}\sim (1+t)^{-\alpha} \quad \mbox{ with } \quad \alpha>\frac{\rs}{2}+\frac 12\,.
 \end{equation}
If the time $t >0$ is sufficiently large so that
\begin{equation}\label{condition1}
\begin{array}{llr}
(1+t)^{\rs-1}&\leq \kappa^{m+\frac 12-\frac d4-\rs} \qquad \mbox{ for } &\alpha\geq \frac 32-\frac{\rs}{2}\\
or\\
(1+t)^{\frac \rs 2-\alpha+\frac 12}&\leq \kappa^{m+\frac 12-\frac d4-\rs} \qquad\mbox{ for }\quad  &\frac \rs 2+\frac 12 < \alpha\leq \frac 32-\frac{\rs}{2}
\end{array}
\end{equation}
hold for some rational number $m\geq d+2$ (which may depend on $\alpha$), then there exists a constant $C>0$ depending on $d, \rs, \|\theta_0\|_2$ and $\alpha$ such that
\begin{equation}\label{estimate1}
 \|\theta(t)\|_2\geq C \kappa^{-\frac d4-\frac{\rs}{2}} (1+t)^{-\frac{d}{4}-\frac{\rs}{2}}\,.
\end{equation}
\end{theorem}
\noindent 
This result does not contradict the energy conservation valid if considering equation \eqref{AD} with $\kappa=0$ for a large class of velocity fields (see Remark \ref{remarkth1}).
In fact notice that the lower bound becomes trivial for $\kappa \rightarrow 0$ since \eqref{condition1} implies $t \rightarrow \infty$, since the exponent of $\kappa$ is positive according to the condition of $m$. 

\bigskip
As mentioned above, the bounds in Theorem \ref{th1} are the key ingredient for the study of $\| \nabla^{-1} \theta\|_2$.
We note that, due to our choice of initial data, $\| \nabla^{-1} \theta_0 \|_2$ is finite in $\R^3$ for $-\frac 12<\rs<1$ and in $\R^2$ for  $0<\rs<1$. 
Using the result in Theorem \ref{th1}, a lower bound on the quantity \eqref{def-H-1} can be obtained ``indirectly'' by using the standard (Gagliardo-Nirenberg) interpolation inequality
$$\|\nabla^{-1}\theta(t)\|_2\geq \frac{\|\theta(t)\|_2^2}{\|\nabla\theta(t)\|_2}\,,$$
if an upper on $\|\nabla \theta(t)\|_2$ can be provided. In Section 3 (see Lemma \ref{UBG}), under a decay assumption on the velocity field of the type 
$
 \|\nabla u(t)\|_{\infty}\sim (1+t)^{-\nu},
$
we derive an upper bound of the form 
\begin{equation}\label{upper-bound-grad-comp}
\|\nabla \theta(t)\|_2\leq C\kappa^{-\frac 12\max\{\frac d4+\frac{\rs}{2}, m\}-1}(1+t)^{-\min\{\frac d4+\frac{\rs}{2}+\frac 12, \frac d4+1\}}f(t)\,,
\end{equation}
where $f(t)$ is a determined function which displays different behavior in time depending on whether $\nu$ is smaller, larger or
equal to one:
\begin{equation*}
\begin{array}{llll}
 f(t) &=& e^{\frac{(1+t)^{-\nu+1}-1}{-\nu+1}} & \mbox{ for } \nu>1\,,\\
 f(t) &=& 1 & \mbox{ for } \nu=1\,,\\
 f(t) &=& (1+t)^{-1}e^{\frac{(1+t)^{-\nu+1}-1}{-\nu+1}}  & \mbox{ for } 0\leq\nu<1\,.
\end{array}
\end{equation*}
Combining \eqref{estimate1} with \eqref{upper-bound-grad-comp} we obtain the following:
\begin{theorem}\label{th2}
 Let the assumptions of Theorem \ref{th1} be satisfied. Additionally suppose that $\rs \in ( 1 - \frac{d}{2}, 1)$
 so that $\|\nabla^{-1} \theta_0\|_2$ is well-defined.
 \begin{itemize}
 \item If 
 \begin{equation}\label{gradient-condition1}
  \|\nabla u(t)\|_{\infty}\sim (1+t)^{-\nu} \mbox{ with } \nu>1\,,
 \end{equation}
 there exists a constant $C>0$ depending on $\|\nabla \theta_0\|_{2}, \|\theta_0\|_2, \rs$ and $d$ such that
  \begin{equation}\label{result1-th2}
     \|\nabla^{-1}\theta(t)\|_2\geq C \kappa^{-\frac d2-\rs+m+\frac 12}e^{-\frac{[(1+t)^{-\nu+1}-1]}{-\nu+1}}(1+t)^{-\frac d4-\frac{\rs}{2}+\frac 12} \,.
  \end{equation}
  \item If
 \begin{equation}\label{gradient-condition2}
  \|\nabla u(t)\|_{\infty}\sim (1+t)^{-1}\,,
 \end{equation}
 there exists a constant $C>0$ depending on $\|\nabla \theta_0\|_{2}, \|\theta_0\|_2, \rs$ and $d$ such that
\begin{equation}\label{result2-th2}
\|\nabla^{-1}\theta(t)\|_2\geq C \kappa^{-\frac d2-\rs+m+\frac 12}(1+t)^{-\frac d4-\frac{\rs}{2}+\frac 12}\, .
\end{equation}
\item If 
 \begin{equation}\label{gradient-condition3}
  \|\nabla u(t)\|_{\infty}\sim (1+t)^{-\nu} \mbox{ with } 0\leq \nu<1\,,
 \end{equation}
 there exists a constant $C>0$ depending on $\|\nabla \theta_0\|_2, \|\theta_0\|_2, \rs$ and $d$ such that
  \begin{equation}\label{result3-th2}
     \|\nabla^{-1}\theta(t)\|_2\geq C \kappa^{-\frac d2-\rs+m+\frac 12}e^{-\frac{[(1+t)^{-\nu+1}-1]}{-\nu+1}}(1+t)^{-\frac d4-\frac{\rs}{2}+\frac 32} \,.
  \end{equation}
\end{itemize}
\end{theorem}
\noindent Let us remark that it is possible to find vector fields that simultaneously satisfy \eqref{assumption-u} and \eqref{gradient-condition1}.
For example, it is known \cite{Fujigaki2001AsymptoticPO}, that the unique strong solution of the Navier-Stokes equation in $\R^3$ with $u_0\in L^3\cap L^1$ and $u_0$ small in $L^3$ satisfies
$$\|u(t)\|_2\leq Ct^{-\frac 54} \qquad \mbox{ and } \qquad \|\nabla u(t)\|_{\infty}\leq Ct^{-\frac 52}\,.$$
Further, it is easy to construct velocity fields which satisfy \eqref{assumption-u} and \eqref{gradient-condition2} or \eqref{gradient-condition3}, respectively, since $u$ does not need to obey any differential equation. 
For example, we can consider the modified two-dimensional shear flow $u=(e^{-\frac{x^2+y^2}{2}}(-y,x)(1+t)^{-\nu},0)$.

Finally, we turn to the study of the filamentation length $\lambda$.
Using again the interpolation inequality, this time written in the form 
 \begin{equation}\label{interpol}
 \frac{\|\nabla^{-1}\theta(t)\|_2}{\|\theta(t)\|_2}\geq \frac{\|\theta(t)\|_2}{\|\nabla\theta(t)\|_2}\,,
 \end{equation}
 together with the upper bound \eqref{upper-bound-grad-comp} and the lower bound \eqref{estimate1}, we find
 \begin{cor}\label{cor}
Under the assumption on $\theta_0$ and $u$ stated in Theorem \ref{th1} and Theorem \ref{th2}, there
exists a constant $C$ depending on $d, \rs, \|\theta_0\|_2, \|\nabla \theta_0\|_2$ and $\alpha$ such that
 \begin{equation}\label{lambda-lb-ad}
 \lambda(t)\geq C(1+t)^{\frac 12}f(t)^{-1}\,,
 \end{equation}
 where 
 $$f=\kappa^{-\frac d4-\frac{\rs}{2}+m+\frac 12}\times
 \begin{cases}
 e^{-\frac{[(1+t)^{-\nu+1}-1]}{-\nu+1}}\,, & \nu>1\\
 1\,, & \nu=1\\
 e^{-\frac{[(1+t)^{-\nu+1}-1]}{-\nu+1}}(1+t)^{-1}\,, &  0<\nu<1\,.
  \end{cases}$$
 \end{cor}
 From this bound we deduce two different asymptotic behaviors: for $t\rightarrow \infty$ the function $(1+t)^{\frac 12}f(t)^{-1}$ goes to infinity for $\nu\geq 1$ while it goes to zero for $\nu\in [0,1)$, indicating \textit{dispersion} in the first case and \textit{mixing} in the second.

\noindent This result does not contradict the observation in \cite{DM2018}: 
In fact our argument seems to suggests that there is no analogous mechanism in the whole space, which enforces the decay of the solution and of its gradient at the same rate. 
Nevertheless, it would be interesting to transfer the approach of this paper to a configuration with bounded domain and periodic boundary conditions as described in \cite{DM2018}.

\bigskip
\subsection*{Notation:} 
In the following results we will denote with $C$ generic constants depending on the data of the problem (initial data, vector field and dimension).
We want to give a fair warning to the reader that in some equations this constant $C$ appears multiple times, but its value might change. Nevertheless, this abuse of notation is motivated by the fact that their exact value is not important for our purposes and we did not attempt to optimize them. Moreover with the symbols $\lesssim$, $\sim$ and $\gtrsim$ we denote the relations $\leq$, $=$ and $\geq$, respectively, hiding numerical constants which may depend on the dimension $d$ and that we do not want to track.

Furthermore $\|\cdot\|_2$ and $\|\cdot\|_{\infty}$ denote the standard $L^2$- and $L^{\infty}$-norms
\begin{itemize}
    \item $\|f\|_2=\left(\int_{\R^d}|f(x)|^2\, dx\right)^{\frac 12}$
    \item $\|f\|_{\infty}=\sup_{x\in \R^d}|f(x)|$\,,
\end{itemize}
and the Fourier transform of $f\in L^2$ is denoted by
$$\hat f(\xi,t)=\int_{\R^d}e^{-i\xi\cdot x}f(x,t)\, dx\,,$$
where $\xi\in \R^d$ is the Fourier-variable.
The $L^2$-norm of the inverse gradient of $f$ is defined as
 \begin{equation*}
  \|\nabla^{-1}f\|_2^2 =\int_{\R^d} (|\xi|^{-1} |\hat{f}(\xi)|)^2\, d\xi\,.
 \end{equation*} 

\subsection*{Organization of the paper:}

The second section is devoted to Theorem \ref{th1}. After stating all the main ingredients for
the result (Lemma \ref{LB-HE} and Lemma \ref{UB-diff}), Theorem \ref{th1} is proved in Subsection 2.1. The
Lemmas, together with the crucial Proposition \ref{pr1}, are subsequently proved in Subsection
2.2. For convenience of the reader, we summarize the steps of the longer proofs (for example
the one of Proposition \ref{pr1}) right at the beginning and verify the steps subsequently. In the
third section, we first state Lemma \ref{UBG}, which is the main tool for the proof of Theorem
\ref{th2}. 
The latter is proved in Subsection 3.1 and the lemma is demonstrated in Subsection 3.2.
Section 4 is devoted to discussion and conclusion.
Finally, in the appendix we compute bounds for the filamentation length for the pure advection equation in the whole space under the same assumptions as in Theorem \ref{th2}, but restricting ourselves to the case $\rs=0$, for simplicity.

\subsection*{Acknowledgement}
We would like to thank Charles Doering and Anna Mazzucato for introducing us to the ``Batchelor-scale problem'', for the various discussions, precious comments and suggestions.
The authors thank the anonymous referees for precious comments and literature suggestions that helped us improving the results in the manuscript. 
CN was partially funded by DFG-GrK2583 and DFG-TRR181.

\section{Theorem \ref{th1}}

We start by splitting the solution of \eqref{AD} into two parts:
Let $T=T(x,t)$ solve the heat equation in $\R^d$
\begin{equation}\label{HE}
\begin{cases}
 \partial_tT=\kappa\Delta T & \mbox{ in } \R^d\times(0,\infty)\,\\
 T(0,x)=\theta_0(x) & \mbox{ in } \R^d\,,
\end{cases}
\end{equation}
then, by subtraction, the function $\eta(x,t) := \theta(x,t)-T(x,t)$ satisfies
\begin{equation}\label{eq:Diff}
\begin{cases}
 \partial_t \eta = \kappa\Delta \eta - u \cdot \nabla \theta & \mbox{ in } \R^d\times(0,\infty)\,\\
 \eta(0,x) = 0 & \mbox{ in } \R^d \, .
\end{cases}
\end{equation}
Observing that $\| \theta(t) \|_2 \geq \| T(t) \|_2 - \| \eta(t) \|_2$, the proof of Theorem \ref{th1} is based on the combination of a lower bound for the solution of the heat equation and an upper bound for $\eta$.

Before stating the crucial lemmas let us give the definition of \textit{decay character} given in \cite{niche2015decay}
\begin{definition}\label{decay-char}
Let $\theta_0\in L^2(\R^d)$. 
The \textit{decay character} of $\theta_0$, denoted by $\rs=\rs(\theta_0)$ is the unique $r\in (-\frac d2,\infty)$ such that 
\begin{equation}\label{decay-indicator}
0<\lim_{\delta\rightarrow 0}\delta^{-2r-d}\int_{|\xi|\leq \delta}|\widehat{\theta}_0(\xi)|^2\, d\xi<\infty
\end{equation}
provided this number exists. 
More compactly, we can define
$$\rs=\rs(\theta_0)=\sup\left\{ r\in \left(-\frac d2,\infty\right) |\lim_{\delta\rightarrow 0}\delta^{-2r-d}\int_{|\xi|\leq \delta}|\hat{\theta}_0|\, d\xi=0\right\}\,.$$
\end{definition}

We restrict our considerations to solutions of \eqref{AD} with algebraic decay and therefore explicitly exclude the cases $r^\ast(\theta_0) = -\frac{d}{2}$ and $r^\ast(\theta_0) = \infty$.

\begin{remark}\label{initial-data}
We observe that $\rs(\theta_0)=\alpha$ for initial data such that $|\hat{\theta}_0(\xi)|\sim |\xi|^{\alpha}$ for $|\xi|\leq \delta$ and $\alpha> - \frac{d}{2}$.
In particular, this implies $\rs(\theta_0)=0$ if $c \leq |\hat\theta_0(\xi)| \leq C$ for $|\xi| \leq \delta$ and some constants $0 \leq c \leq C$.
Notice, however, that this class also contains mean-free initial data in real space, since $\hat{\theta}_0(\xi=0)=0$ for $\alpha > 0$.
Another example for which the limit \eqref{decay-indicator} exists was given in \cite[Example 2.6]{niche2015decay} (and in \cite[Example 2.6]{FNP16}, where it was slightly corrected): if $\theta_0\in L^p(\R^d)\cap L^2(\R^d)$ with $1\leq p\leq 2$, then $\rs(\theta_0)=-d(1-\frac 1p)$.
\end{remark}
\begin{remark}\label{re:BrandoleseRS}
The limit in \eqref{decay-indicator} might not exist.
Brandolese \cite{B16} explicitly constructs initial data $\theta_0\in L^2(\R^d)$ with very fast oscillations near the origin, for which the limit is not well defined.
The author then proceeds to relax the requirements for the existence of decay characters, giving a new (more general) definition, which also allow such initial data (excluded from the previous theory).
In the same paper, the initial data, for which decay characters in the new definition exist, is characterized in terms of subsets of Besov spaces.
With this, it is proved that the solution of $\partial_t u=\mathcal L u$ with initial data $u_0 \in L^2$, where $\mathcal L$ is a pseudo-differential operator with homogeneous symbol, satisfies
\begin{equation}\label{BB}
(1+t)^{-\frac{1}{\alpha}(\rs+\frac d2)}\lesssim\|u(t)\|_2\lesssim (1+t)^{-\frac{1}{\alpha}(\rs+\frac d2)}\,,
\end{equation}
(with $\alpha$ depending on the symbol of $\mathcal{L})$ if and only if $u_0\in L^2$ is such that the decay character $\rs(u_0)\in(-\frac d2, \infty)$ exists.
\end{remark}

\begin{lemma}[Lower bound for the solution of \eqref{HE}]\label{LB-HE}
Let $T$ solve \eqref{HE} and $\theta_0\in L^2(\R^d)$ have decay character $\rs(\theta_0)=\rs$.
 If $-\frac d2<\rs<\infty$
 then there exists a constant $C>0$ depending on $\kappa$ and $\rs$ such that
 \begin{equation}\label{lbh}
 \|T(t)\|_2\geq C (\kappa(1+t))^{-\frac d4-\frac{\rs}{2}}\,.
 \end{equation}
\end{lemma}
\begin{lemma}[Upper bound for the solution of \eqref{eq:Diff}]\label{UB-diff}
Consider $\eta=\theta-T$, solution of equation \eqref{eq:Diff}. 
Let $\theta_0\in L^2(\R^2)$ be the initial condition with decay character $\rs \in (-\frac d2, \infty)$ 
 and let $u(\cdot,t)\in L^2(\R^d)$ be a divergence-free vector field such that 
 \begin{equation*}\label{new-ass-u-diff}
 \|u(t)\|_2\sim (1+t)^{-\alpha}  \;\mbox{ for some } \alpha > \max\left\{\frac{1}{2} - \frac{d}{4}, 1-\frac d4-\frac{\rs}{2}\right\} \, .
 \end{equation*}
Then there exist a rational number $m\geq \frac d2+1$ and a constants $C>0$ depending on $d,\rs,\|\theta_0\|_2$ and $\alpha$, such that 
\begin{equation}\label{ubound-eta}
\begin{array}{lrll}
&\|\eta(t)\|_2^2&\leq C
\kappa^{-m-\frac d4-\frac 12}(1+t)^{-\min\{\frac d2+1,\frac d2+\frac{\rs}{2}+\alpha - \frac{1}{2}\}}& \rs \leq 1 
\\
&\|\eta(t)\|_2^2&\leq C
\kappa^{-m-\frac d4-\frac 12}(1+t)^{-\min\{\frac d2+1, \frac d2+\alpha\}} & \rs\geq 1 
\end{array}
\end{equation}
holds.
\end{lemma}
The lower bound for $\theta=T+\eta$ now follows easily from Lemma \ref{LB-HE} and Lemma \ref{UB-diff}.
\subsection{Proof of Theorem \ref{th1}}
\begin{proof}[Proof of Theorem \ref{th1}]

Combining estimate \eqref{lbh} with \eqref{ubound-eta} we obtain: 
 \begin{itemize}
 \item 
 For $\rs< 1$, $\alpha\geq\frac 32-\frac{\rs}{2}$
 \begin{align*}
 \|\theta(t)\|_2^2 & \geq \|T(t)\|_2^2-\|\eta(t)\|_2^2 \\
  & \geq C \kappa^{-\frac d2-\rs} (1+t)^{-\frac{d}{2}-\rs}\left[1-\kappa^{-m-\frac 12+\frac d4+\rs}(1+t)^{\rs-1}\right]\,.
 \end{align*}
 Thus for $t>0$ sufficiently large such that $\kappa^{-m-\frac 12+\frac d4+\rs}(1+t)^{\rs-1}<1$ we have
 \begin{equation*}
 \|\theta(t)\|_2\geq C \kappa^{-\frac d2-\rs} (1+t)^{-\frac{d}{2}-\rs}\,.
 \end{equation*}
 \item For $\rs< 1$, $\frac{\rs}{2}+\frac 12<\alpha\leq\frac 32-\frac{\rs}{2}$
  \begin{align*}
  \|\theta(t)\|_2 & \geq \|T(t)\|_2-\|\eta(t)\|_2 \\
  & \geq C \kappa^{-\frac d2-\rs} (1+t)^{-\frac{d}{2}-\rs}\left[1-\kappa^{-m-\frac 12+\frac d4+\rs}(1+t)^{\frac{\rs}{2}-\alpha+\frac 12}\right]\,.
 \end{align*}  
 Thus for $t>0$ sufficiently large such that $\kappa^{-m-\frac 12+\frac d4+\rs}(1+t)^{\frac{\rs}{2}-\alpha+\frac 12}<1$ we have
 \begin{equation*}
 \|\theta(t)\|_2\geq C\kappa^{-\frac d2-\rs}  (1+t)^{-\frac{d}{2}-\rs}\,.
 \end{equation*}
\end{itemize}
 Instead in the regime $\rs\geq 1$ we do not get any lower bounds as (the upper bound on) the energy of $\eta$ decays to zero slower than diffusion (see \eqref{ubound-eta}) and therefore the difference cannot be positive.
 In fact, this regime seems to be penalized by the adopted perturbation argument.
\end{proof}
\begin{remark}\label{remarkth1}
 Observe that for the pure advection equation (setting $\kappa=0$ in \eqref{AD}) we have $\|\theta(t)\|_2=\|\theta_0\|_2$ for sufficiently regular $u$ (for example when $u$ is smooth or in the DiPerna-Lions class). This does not contradict our result: In fact, passing the limit $\kappa\rightarrow 0$ in our result, we see that the conditions of validity above are not satisfied for finite times. 
\end{remark}

\subsection{Proof of Lemmas}
The result in Lemma \ref{LB-HE} is already proved in \cite[Theorem 2.10]{niche2015decay}. For convenience of the reader we report its proof here.
\begin{proof}[Proof of Lemma \ref{LB-HE}]
Because of condition \eqref{decay-indicator}, there exists a $\delta_0>0$ and $C_1>0$ such that if $0<\delta\leq\delta_0$ (to be chosen later) we have
$$C_1\delta^{2r^\ast +d}< \int_{|\xi|\leq \delta}|\hat{\theta}_0(\xi)|^2\, d\xi $$
 By Plancherel's theorem and the assumptions on $\theta_0$, for some $\delta=\delta(t)$ we have 
\begin{eqnarray*}
 \int_{\R^d}|T(x)|^2\, dx
          &=&\int_{\R^d}|\hat T(\xi)|^2\, d\xi\\
          &=&\int_{\R^d}|\hat \theta_0(\xi)|^2e^{-2\kappa|\xi|^2 t}\, d\xi\\
          &\geq&\int_{|\xi|\leq \delta(t)}|\hat \theta_0(\xi)|^2e^{-2\kappa|\xi|^2 t}\, d\xi\\
          &\geq&e^{-2\kappa\delta^2 t}\int_{|\xi|\leq \delta(t)}|\hat \theta_0(\xi)|^2\, d\xi\\
           &>&e^{-2\kappa\delta^2 t}C_1\delta^{2r^\ast +d}\\
\end{eqnarray*}
Setting $\delta(t):=\delta_0(\kappa(1+t))^{-\frac 12}$, then $e^{-2\kappa\delta^2 t}= e^{-2\kappa\delta_0^2(\kappa(1+t))^{-1} t}\geq C>0 $ and
\begin{equation*}
 \|T(t)\|_2^2\geq C (\kappa(1+t))^{-r^\ast-\frac d2}\,.
\end{equation*}
\end{proof}
In order to prove Lemma \ref{UB-diff} we need the following
\begin{prop}[Upper bound for the solution to \eqref{AD}]\label{pr1}
 Let $d=2,3$, $\theta_0\in L^2(\R^d)$ be the initial condition with decay character $\rs \in(-\frac d2, \infty)$ 
 and let $u(\cdot,t)\in L^2(\R^d)$ be a divergence-free vector field such that 
 \begin{equation}\label{new-assumption-u}
 \|u(t)\|_2 \sim (1+t)^{-\alpha} \quad\mbox{for some} \quad \alpha > \frac{1}{2} - \frac{d}{4} .
 \end{equation}
 Then there exists a positive constant $C$ depending on $d,\rs\|\theta_0\|_2$ and $\alpha$ such that
  \begin{equation}\label{result:ub}
   \|\theta(t)\|_2\leq C \kappa^{-\max\{\frac d4+\frac{\rs}{2}, m\}}(1+t)^{ -\min\{\frac d4+\frac{\rs}{2}, \frac d4+\frac 12\} }
  \end{equation}
  for some $m \geq \frac{d}{2}+1$ depending on $\ceil{ \frac{1}{|\alpha|}}$.
\end{prop} 
    \begin{remark}
    Notice that, combining this upper bound with the lower bound in Theorem \ref{th1}, we find that, for $\rs < 1$ and $\alpha>\frac{\rs}{2}+\frac 12$, our result is sharp, i.e.
    $$\|\theta(t)\|_{2}\sim (\kappa(1+t))^{-\frac{d}{4}-\frac{\rs}{2}}\,.$$
    \end{remark}
\begin{proof}
We divide the proof of Proposition \ref{pr1} in four steps which, for convenience of the reader, we state first and verify afterwards.
\begin{description}
 \item[Step 1] Define the set 
 \begin{equation}\label{def:S}
 S(t)=\left\{\xi\in \R^d\; |\quad  |\xi|\leq \left(\frac{\beta}{2\kappa(1+t)}\right)^{\frac 12}\right\}\,.
 \end{equation}
 Passing through the energy identity for equation \eqref{AD}, we have 
  \begin{equation}\label{eq:ode2}
\frac{d}{dt}((1+t)^{\beta}\|\theta(t)\|_2^2)\leq \beta(1+t)^{\beta-1}\int_{S(t)}|\hat{\theta}(\xi,t)|^2\, d\xi\,.
\end{equation}
\item[Step 2]
 Under assumption \eqref{new-assumption-u} for $u$ and using the bound for the heat kernel
  \begin{equation}\label{first-part}
   \int_{S(t)} e^{-2\kappa|\xi|^2t}|\hat\theta_0(\xi)|^2\, d\xi\leq C  (\kappa(1+t))^{-\frac d2-\rs}\,,
  \end{equation}
  (the proof of this estimate can be found in \cite[Theorem 2.10]{niche2015decay})
 %
 we obtain
  \begin{multline}\label{eq:ode3}
   \int_{S(t)}|\hat\theta(\xi,t)|^2\,d\xi
  \leq 2C\kappa^{-\frac d2-\rs}(1+t)^{-\frac d2-\rs}\\
  +\frac{2}{2+d}\left(\frac{\beta}{2\kappa(1+t)}\right)^{\frac{2+d}{2}}\,t\int_0^t \|\theta(t)\|_2^2\|u(t)\|_2^2\, ds. 
  \end{multline}
  In particular, the combination with \eqref{eq:ode2} yields
  \begin{multline}\label{eq:ode4}
    \frac{d}{dt}((1+t)^{\beta}\|\theta(t)\|_2^2)\leq 2C\kappa^{-\frac d2-\rs}(1+t)^{-\frac d2-\rs+\beta-1}\\
    +\frac{2}{2+d}\beta(1+t)^{\beta-1}\left(\frac{\beta}{2\kappa(1+t)}\right)^{\frac{2+d}{2}}\,t\int_0^t \|\theta(t)\|_2^2\|u(t)\|_2^2\, ds\,.
\end{multline}

  \item[Step 3]
   Estimating the second term on the right-hand side of \eqref{eq:ode4} we have the upper bounds 
   \begin{equation*}
   \|\theta\|_2^2\leq C \kappa^{-\max\{\frac d2+\rs, \frac d2+1\}}(1+t)^{ -\min\{\frac d2+\rs, \frac d2\} }\,
   \end{equation*}
   where the constant $C$ depends on $d, \rs$ and $\|\theta_0\|_2$.

 \item[Step 4] 
By iterating over the effect of \eqref{new-assumption-u} we obtain
\begin{equation*}
\|\theta\|_2^2\leq C \kappa^{-\max\{\frac d2+\rs, 2m\}}(1+t)^{ -\min\{\frac d2+\rs, \frac d2+1\} }\,
\end{equation*}
where $m>0$ depends on the number of iteration needed, proportional to $\ceil{\frac{1}{|\alpha|}}$. 
\end{description}
\textbf{Proof of Step 1: } 
 We start by testing equation \eqref{AD} with $\theta$, integrating by parts and, using the incompressibility condition, obtaining
\begin{equation*}
 \frac{d}{dt}\|\theta(t)\|_2^2=-2\kappa \|\nabla \theta(t)\|_2^2\,,
\end{equation*}
 which we can rewrite in Fourier space using Plancherel's theorem 
\begin{equation}\label{fs-rep}
\frac{d}{dt}\|\hat{\theta}(t)\|_2^2=-2\kappa \|\xi\,\hat{\theta}(t)\|_2^2\,.
\end{equation}
Consider the set
\begin{equation}\label{eq:S}
S(t):=\{\xi\in  \R^d\; |\;  |\xi|\leq R(t)\}\,,
\end{equation}
where $R(t)$ will be specified later, and split the integral on the right-hand-side of \eqref{fs-rep} over $S(t)$ and its complement $S^c(t)$. Using the positivity of the integrands and the definition of the set $S(t)$ we have
\begin{eqnarray*}
    \frac{d}{dt}\|\hat{\theta}(t)\|_2^2                               &=&-2\kappa\int_{S(t)}|\xi|^2|\hat{\theta}(\xi,t)|^2\, d\xi-2\kappa\int_{S^c(t)}|\xi|^2|\hat{\theta}(\xi,t)|^2\, d\xi\\
                                   &\leq&-2\kappa\int_{S^c(t)}|\xi|^2|\hat{\theta}(\xi,t)|^2\, d\xi\\
                                   &\leq&-2\kappa R^2(t)\int_{S^c(t)}|\hat{\theta}(\xi,t)|^2\, d\xi\\
                                   &= &-2\kappa R^2(t)\int_{\R^d}|\hat{\theta}(\xi,t)|^2\, d\xi+2\kappa R^2(t)\int_{S(t)}|\hat{\theta}(\xi,t)|^2\, d\xi\,.
\end{eqnarray*}
Choose 
\begin{equation}\label{eq:r}
R^2(t)=\frac{\phi '(t)}{2\kappa\phi(t)}\qquad \mbox{ with } \phi:\R^+\rightarrow \R \mbox{ increasing} 
\end{equation}
so that we can rewrite the above estimate as
\begin{equation}\label{eq:ode1}
\frac{d}{dt}(\phi(t)\|\theta(t)\|_2^2)\leq \phi'(t)\int_{S(t)}|\hat{\theta}(\xi,t)|^2\, d\xi\,.
\end{equation}
Defining
$$\phi(t)= (1+t)^{\beta}$$ with $\beta>0$ (to be chosen at the end), we obtain \eqref{eq:ode2}. Moreover the expression of $R$ can now be determined explicitly from \eqref{eq:r}: 
\begin{equation}\label{eq:r-explicit}
	R^2(t)=\frac{\phi'(t)}{2\kappa\phi(t)}=\frac{1}{2\kappa}\frac{d}{dt}\log \phi(t)=\frac{1}{2\kappa}\beta\frac{d}{dt} \ln(1+t)=\frac{\beta}{2\kappa(1+t)}\,.
\end{equation}
%
 \noindent\textbf{Proof of Step 2}:
 Write equation \eqref{AD} in Fourier space 
\begin{equation*}
 \partial_t\hat\theta(\xi)+\kappa|\xi|^2\hat\theta(\xi)=-\widehat{u\cdot \nabla\theta}(\xi)\,,
\end{equation*}
and the representation formula for its solution
\begin{equation*}
 \hat \theta(\xi,t)=e^{-\kappa|\xi|^2t}\hat\theta_0(\xi)+\int_0^te^{-\kappa |\xi|^2(t-s)}(-\widehat{u\cdot \nabla \theta})(\xi,s)\, ds\,.
\end{equation*}
Squaring, applying the Young Inequality $ab\leq\frac 12 a^2+\frac 12 b^2$ 
and integrating over $S(t)$, we obtain 
\begin{multline}\label{tbe}
\int_{S(t)}|\hat\theta(\xi,t)|^2\,d\xi
\leq 2\int_{S(t)}|e^{-\kappa |\xi|^2t}\hat\theta_0|^2\, d\xi
+2\int_{S(t)}\left(\int_0^t e^{-\kappa|\xi|^2 (t-s)}|\widehat{u\cdot \nabla \theta}|\, ds\right)^2\,d\xi. 
\end{multline}
Next, we estimate the right-hand side of \eqref{tbe}: for the first term we apply the heat-kernel estimate \eqref{first-part}.
 The claim in Step 2 is achieved by estimating the product $\widehat{u\cdot \nabla \theta}$ using the definition of Fourier transform and the assumptions on $u$ and $\theta$:
 \begin{align}\label{pivot}
  \begin{split}
   |\widehat{u\cdot\nabla \theta}|
   &=\left|\int u(x,t)\cdot\nabla \theta(x,t) e^{-i\xi\cdot x }\, dx\right|\\
   &=\left|\int \nabla\cdot(u(x,t) \theta(x,t)) e^{-i\xi\cdot x }\, dx\right|\\
   &=\left|\int u(x,t) \theta(x,t) i\xi e^{-i\xi\cdot x }\, dx\right|\\
   &\leq|\xi|\|\theta(t)\|_2\|u(t)\|_2\,.\\
  \end{split}
  \end{align}
  So, we have
  \begin{multline*}
\int_{S(t)}|\hat\theta(\xi,t)|^2\,d\xi
\leq 2\int_{S(t)}|e^{-\kappa |\xi|^2t}\hat\theta_0|^2\, d\xi
+2\int_{S(t)}|\xi|^2\, d\xi\,t\int_0^t \|\theta(t)\|_2^2\|u(t)\|_2^2\, ds. 
\end{multline*}
where we used $|e^{-\kappa|\xi|^2(t-s)}|\leq 1$ and the Cauchy-Schwarz inequality.
Passing to polar coordinates we compute the integral
\begin{equation*}
\int_{S(t)}|\xi|^2\, d\xi\sim \frac{1}{2+d}  R(t)^{2+d}= \frac{1}{2+d}\left(\frac{\beta}{2\kappa(1+t)}\right)^{\frac{2+d}{2}}\,.
\end{equation*}
Hence we obtain
\begin{multline*}
\int_{S(t)}|\hat\theta(\xi,t)|^2\,d\xi
\leq 2C\kappa^{-\frac d2-\rs}(1+t)^{-\frac d2-\rs}\\
+\frac{2}{2+d}\left(\frac{\beta}{2\kappa(1+t)}\right)^{\frac{2+d}{2}}\,t\int_0^t \|\theta(t)\|_2^2\|u(t)\|_2^2\, ds. 
\end{multline*}


 \noindent \textbf{Proof of Step 3}:
  Integrating \eqref{eq:ode4} between $0$ and $t$ 
  \begin{multline*}
    (1+t)^{\beta}\|\theta(t)\|_2^2\leq \|\theta_0\|_2^2 +\frac{2C}{(-\frac d2-\rs+\beta)}\kappa^{-\frac d2-\rs}(1+t)^{-\frac d2-\rs+\beta}\\
    +\frac{2}{2+d}\beta\left(\frac{\beta}{2}\right)^{\frac{2+d}{2}}\frac{1}{(-\frac d2+\beta)}\kappa^{-\frac d2-1}(1+t)^{-\frac d2+\beta}\int_0^t \|\theta(t)\|_2^2\|u(t)\|_2^2\, ds
\end{multline*}
and then, dividing by $(1+t)^{-\frac d2+\beta}$, we get
  \begin{multline*}
    (1+t)^{\frac d2}\|\theta(t)\|_2^2\leq \|\theta_0\|_2^2(1+t)^{\frac d2-\beta} +\frac{C_1}{(-\frac d2-\rs+\beta)}\kappa^{-\frac d2-\rs}(1+t)^{-\rs}\\
    +C_2\frac{1}{(-\frac d2+\beta)}\kappa^{-\frac d2-1}\int_0^t (1+t)^{\frac d2}\|\theta(t)\|_2^2(1+t)^{-\frac d2}\|u(t)\|_2^2\, ds \,,
\end{multline*}
where we smuggled in the weight $(1+t)^{\frac d2}$ in the integral on the right-hand side and where we set
$C_1=2C$ and 
$C_2=\frac{2}{2+d}\beta\left(\frac{\beta}{2}\right)^{\frac{2+d}{2}}$.

Now set 
\begin{eqnarray*}
X(t)&=&(1+t)^{\frac d2}\|\theta\|_2^2\\
a(t)&=&\frac{C_2}{(-\frac d2+\beta)}\kappa^{-\frac d2-1}(1+t)^{-\frac d2}\|u(t)\|_2^2 \\
b(t)&=& \|\theta_0\|_2^2(1+t)^{\frac d2-\beta} +\frac{C_1}{(-\frac d2-\rs+\beta)}\kappa^{-\frac d2-\rs}(1+t)^{-\rs} \,
\end{eqnarray*}
so that the previous bound can be written in the compact form 
$$X(t)\leq b(t)+\int_0^t a(s) X(s)\, ds\,.$$
We need to distinguish two cases, depending on whether $b(t)$ is an increasing or decreasing function of time.
\begin{enumerate}

\item If $\rs\leq 0$ then
$$X(t)\leq b(t)\exp\left(\int_0^t a(\tau)\, d\tau\right)\,,$$
that is 
\begin{multline*}
(1+t)^{\frac d2}\|\theta\|_2^2 \leq \left(\|\theta_0\|_2^2(1+t)^{\frac d2-\beta}+\frac{C_1}{(-\frac d2-\rs+\beta)}\kappa^{-\frac d2-\rs}(1+t)^{-\rs}\right) \times \\
\times \exp\left(\int_0^t \frac{C_2}{(-\frac d2+\beta)}\kappa^{-\frac d2-1}(1+\tau)^{-\frac d2}\|u(\tau)\|_2^2\, d\tau\right)\,.
\end{multline*}
According to our assumption \eqref{new-assumption-u}, $\int_0^{\infty}(1+\tau)^{-\frac d2}\|u(\tau)\|_2^2\, d\tau<\infty$, and we can estimate 
\begin{align*}
(1+t)^{\frac d2}\|\theta\|_2^2
\leq C \left(\|\theta_0\|_2^2(1+t)^{\frac d2-\beta}+\frac{C_1}{(-\frac d2-\rs+\beta)}\kappa^{-\frac d2-\rs}(1+t)^{-\rs}\right)\,.
\end{align*}
Therefore 
\begin{align*}
\|\theta\|_2^2
\leq C \left(\|\theta_0\|_2^2(1+t)^{-\beta}+\frac{C_1}{(-\frac d2-\rs+\beta)}\kappa^{-\frac d2-\rs}(1+t)^{-\frac d2-\rs}\right)\,
\end{align*}
and choosing $\beta>\frac d2+\rs$ we have 
\begin{align*}
\|\theta\|_2^2
\leq C\kappa^{-\frac d2-\rs}(1+t)^{-\frac d2-\rs}\,.
\end{align*}
where $C$ depends on $d,\rs,\|\theta_0\|_2$. 
\item If $\rs \geq 0$ then 
 $$X(t)\leq b(t)+\int_0^t\,b(s)\,a(s)\exp\left(\int_s^t a(\tau)\, d\tau\right)\, ds\,,$$
 that is 
\begin{align*}
 (1+t)^{\frac d2}\|\theta\|_2^2 \leq&
 \|\theta_0\|_2^2(1+t)^{\frac d2-\beta}+\frac{C_1}{(-\frac d2-\rs+\beta)}\kappa^{-\frac d2-\rs}(1+t)^{-\rs}\\
 &+\int_0^t \left[\left(\|\theta_0\|_2^2(1+s)^{\frac d2-\beta} +\frac{C_1}{(-\frac d2-\rs+\beta)}\kappa^{-\frac d2-\rs}(1+s)^{-\rs}\right)\right.\\
 & \quad \times \left.\frac{C_2}{(-\frac d2+\beta)}\kappa^{-\frac d2-1}(1+s)^{-\frac d2}\|u(s)\|_2^2\right.\,\\
 & \quad \times \left.\exp\left(\int_s^t \frac{C_2}{(-\frac d2+\beta)}\kappa^{-\frac d2-1}(1+\tau)^{-\frac d2}\|u(\tau)\|_2^2\, d\tau\right)\right]\, ds\,.  
 \end{align*}
 Dividing by $(1+t)^{\frac d2}$
 \begin{align*}
 \|\theta\|_2^2\leq&
 \|\theta_0\|_2^2(1+t)^{-\beta}+\frac{C_1}{(-\frac d2-\rs+\beta)}\kappa^{-\frac d2-\rs}(1+t)^{-\frac d2-\rs}\\
 &+(1+t)^{-\frac d2}\int_0^t \left[\left(\|\theta_0\|_2^2(1+s)^{\frac d2-\beta} +\frac{C_1}{(-\frac d2-\rs+\beta)}\kappa^{-\frac d2-\rs}(1+s)^{-\rs}\right)\right.\\
 & \quad \times \left.\frac{C_2}{(-\frac d2+\beta)}\kappa^{-\frac d2-1}(1+s)^{-\frac d2}\|u(s)\|_2^2\right.\,\\
 & \quad \times \left.\exp\left(\int_s^t \frac{C_2}{(-\frac d2+\beta)}\kappa^{-\frac d2-1}(1+\tau)^{-\frac d2}\|u(\tau)\|_2^2\, d\tau\right)\right]\, ds\,.  
 \end{align*}
Notice that $(1+\tau)^{-\frac d2}\|u(\tau)\|_2^2$ is integrable between $s$ and $\infty$ if
 \begin{equation}\label{cond-velocity}
 (1+\tau)^{-\frac d2}\|u(\tau)\|_2^2\leq c(1+\tau)^{-1-\varepsilon}\,,
 \end{equation}
 i.e. $\| u(\tau) \|_2 \leq c (1 + \tau)^{-\alpha}$ with $\alpha > \frac{1}{2} - \frac{d}{4}$, 
 so we can write 
 \begin{align*}
 \|\theta\|_2^2\leq&
 \|\theta_0\|_2^2(1+t)^{-\beta}+\frac{C_1}{(-\frac d2-\rs+\beta)}\kappa^{-\frac d2-\rs}(1+t)^{-\frac d2-\rs}\\
 &+C_0(1+t)^{-\frac d2}\int_0^t \left[\left(\|\theta_0\|_2^2(1+s)^{\frac d2-\beta} +\frac{C_1}{(-\frac d2-\rs+\beta)}\kappa^{-\frac d2-\rs}(1+s)^{-\rs}\right)\right.\\
 & \quad \times \left.\frac{cC_2}{(-\frac d2+\beta)}\kappa^{-\frac d2-1}(1+s)^{-1-\varepsilon}\right]\,ds
 \end{align*}
 where we used that for all $t\geq 0$ there exists a positive constant $C_0$
 $$\exp\left(\int_s^{\infty} \frac{cC_2}{(-\frac d2+\beta)}\kappa^{-\frac d2-1}(1+\tau)^{-1-\varepsilon}\, d\tau\right)\leq C_0\,.$$
 Integrating the right hand-side of 
 {\small\begin{align*}
 \|\theta\|_2^2\leq &
 \|\theta_0\|_2^2(1+t)^{-\beta}+\frac{C_1}{(-\frac d2-\rs+\beta)}\kappa^{-\frac d2-\rs}(1+t)^{-\frac d2-\rs}\\
 & + C_0(1+t)^{-\frac d2}\left[\|\theta_0\|_2^2\frac{cC_2}{(-\frac d2+\beta)}\kappa^{-\frac d2-1}\int_0^t(1+s)^{\frac d2-\beta-1-\varepsilon}\, ds\right.\\
 & + \left.\frac{cC_1C_2}{(-\frac d2-\rs+\beta)(-\frac d2+\beta)}\kappa^{-d-\rs-1}\int_0^t\,(1+s)^{-\rs-1-\varepsilon} ds\right]
 \end{align*}}
 in time, we obtain
 {\small\begin{align*}
 \|\theta\|_2^2\leq &
 \|\theta_0\|_2^2(1+t)^{-\beta}+\frac{C_1}{(-\frac d2-\rs+\beta)}\kappa^{-\frac d2-\rs}(1+t)^{-\frac d2-\rs}\\
 &+C_0(1+t)^{-\frac d2}\left\{\|\theta_0\|_2^2\frac{cC_2}{(-\frac d2+\beta)(\frac d2-\beta-\varepsilon)}\kappa^{-\frac d2-1}[(1+t)^{\frac d2-\beta-\varepsilon}-1]\right.\\
 &+\left.\frac{cC_1C_2}{(-\frac d2-\rs+\beta)(-\frac d2+\beta)(-\rs-\varepsilon)}\kappa^{-d-\rs-1}[(1+t)^{-\rs-\varepsilon}-1] \right\}
 \end{align*}}
 We choose $\beta>\frac d2+\rs$ and, since $\rs\geq 0$, we estimate
 \small\begin{multline*}
 \|\theta_0\|_2^2\frac{cC_2}{(-\frac d2+\beta)(\frac d2-\beta-\varepsilon)}
 [(1+t)^{\frac d2-\beta-\varepsilon}-1]\\
 +\frac{cC_1C_2}{(-\frac d2-\rs+\beta)(-\frac d2+\beta)(-\rs-\varepsilon)}
 [(1+t)^{-\rs-\varepsilon}-1] \leq C_3\,,
 \end{multline*}
 so that
 \begin{align*}
 \|\theta\|_2^2 \leq &
 \|\theta_0\|_2^2(1+t)^{-\beta}+\frac{C_1}{(-\frac d2-\rs+\beta)}\kappa^{-\frac d2-\rs}(1+t)^{-\frac d2-\rs}+C_4\kappa^{-\frac d2-1}(1+t)^{-\frac d2}\\
 \leq & C \kappa^{-\frac d2-1}(1+t)^{-\frac d2}\,.
 \end{align*}
 where $C$ depends on $\rs,d$ and $\|\theta_0\|_2$.
\end{enumerate}
\noindent \textbf{Proof of Step 4}:
We look at the region $\rs\geq 0$ and improve the result by iteration. From the previous step we have
$$\|u(t)\|_2\sim (1+t)^{-\alpha} \mbox{ with } \alpha > \frac{1}{2} - \frac{d}{4} \quad \mbox{ and } \quad \|\theta(t)\|_2\leq C \kappa^{-\frac d4-\frac 12}(1+t)^{-\frac d4}\,.$$
Starting again from \eqref{pivot} we have
\begin{align*}
  \begin{split}
   |\widehat{u\cdot\nabla \theta}|
   &\leq|\xi|\|\theta(t)\|_2\|u(t)\|_2\,\\
   &\leq C\kappa^{-\frac d4-\frac 12}|\xi|(1+t)^{-\frac d4- \alpha }\,.
  \end{split}
  \end{align*}
Then, using $|e^{-\kappa |\xi|^2(t-s)}|\leq 1$ we have
  \begin{equation}\label{here1}
   \int_0^t e^{-\kappa |\xi|^2(t-s)} C \kappa^{-\frac d4-\frac 12}|\xi|(1+s)^{-\frac d4- \alpha}\, ds\leq  \frac{C}{(-\frac d4- \alpha + 1)}\kappa^{-\frac d4-\frac 12} |\xi|(1+t)^{-\frac d4- \alpha + 1}
  \end{equation}
  for $\alpha < 1 - \frac{d}{4}$. 
 Integrate over $S(t)$:
  \begin{multline*}
   \int_{S(t)}2\left( \int_0^t e^{-\kappa |\xi|^2 (t-s)} |\widehat{u\cdot\nabla \theta}|\,ds\right)^2\, d\xi 
   \leq \int_{S(t)}|\xi|^2\, d\xi\, \frac{C^2}{(-\frac d4- \alpha + 1)^2}\kappa^{-\frac d2-1} (1+t)^{2-\frac{d}{2} - 2 \alpha}
  \end{multline*}
  Recalling the computation in Step 2
  \begin{equation}\label{triangle}
  \int_{S(t)}|\xi|^2\, d\xi\sim \frac{1}{2+d}\left(\frac{\beta}{2\kappa (1+t)}\right)^{\frac{2+d}{2}}\,.
  \end{equation}
 and inserting it in the previous bound, we obtain
  \begin{multline*}
   \int_{S(t)}2\left( \int_0^t e^{-\kappa |\xi|^2 (t-s)} |\widehat{u\cdot\nabla \theta}|\,ds\right)^2\, d\xi \\
   \leq \frac{1}{2+d}\left(\frac{\beta}{2}\right)^{\frac{2+d}{2}} \frac{C^2}{(-\frac d4- \alpha + 1)^2}\kappa^{-d-2}  (1+t)^{1 - d - 2 \alpha}
  \end{multline*}
As a result, inserting this estimate in the second term of the right-hand side of \eqref{tbe} we find
\begin{align*}
\int_{S(t)}|\hat\theta(\xi,t)|^2\,d\xi
&\leq 2\int_{S(t)}|e^{-\kappa |\xi|^2t}\hat\theta_0|^2\, d\xi\\
& \quad + 2\int_{S(t)}\left(\int_0^t e^{-\kappa|\xi|^2 (t-s)}|\widehat{u\cdot \nabla \theta}|\, ds\right)^2\,d\xi\\
&\leq 2C\kappa^{-\frac d2-\rs}(1+t)^{-\frac d2-\rs}\\
& \quad +\frac{1}{2+d}\left(\frac{\beta}{2}\right)^{\frac{2+d}{2}} \frac{A^2}{(-\frac d4- \alpha + 1)^2}\kappa^{-d-2}  (1+t)^{1 - d - 2 \alpha}\,.
\end{align*}
\footnote{We warn the reader that the two constants $C$ appearing on the right-hand side of the bound are not the same. This abuse of notation is justified at the end of the introduction. }
Finally, inserting in \eqref{eq:ode2} we obtain 
\begin{align*}
\frac{d}{dt}((1+t)^{\beta}\|\theta(t)\|_2^2) &
\leq \beta(1+t)^{\beta-1}\int_{S(t)}|\hat{\theta}(\xi,t)|^2\, d\xi\\
&\leq\beta2C\kappa^{-\frac d2-\rs}(1+t)^{-\frac d2-\rs+\beta-1}\\
& \quad +\beta\frac{1}{2+d}\left(\frac{\beta}{2}\right)^{\frac{2+d}{2}} \frac{C^2}{(-\frac d4- \alpha + 1)^2}\kappa^{-d-2}  (1+t)^{\beta - d - 2 \alpha }\,.
\end{align*}
Then, choosing $\beta>\max\{\frac d2+\rs, d + 2 \alpha - 1 \}$ 
\begin{align*}
\|\theta\|_2^2&\leq (1+t)^{-\beta}\|\theta_0\|_2^2+\frac{2\beta C}{(-\frac d2-\rs+\beta)}\kappa^{-\frac d2-\rs}(1+t)^{-\frac d2-\rs}\\
& \quad +\beta\frac{1}{2+d}\left(\frac{\beta}{2}\right)^{\frac{2+d}{2}}\frac{C^2}{(-\frac d4- \alpha + 1)^2 (\beta - d - 2\alpha + 1)}\kappa^{-d-2}(1+t)^{-d - 2\alpha + 1}\\
& \leq C \kappa^{-\max\{\frac d2+\rs, d+2\}}(1+t)^{ -\min\{\frac d2+\rs, d + 2\alpha - 1\} }\,,
\end{align*}
where the constant $C>0$ depends on $d,\rs$ and $\alpha$.

Note that for this result $\frac{1}{2} - \frac{d}{4} < \alpha < 1 - \frac{d}{4}$ holds, which implies that we gained a better decay depending on $\alpha$. 
The decay can be improved by iterating this argument $\ceil{\frac{1}{|\alpha|}}$ times obtaining
\begin{equation}\label{eq:propositionFinalResult}
\|\theta\|_2^2\leq C \kappa^{-\max\{\frac d2+\rs, 2m\}}(1+t)^{ -\min\{\frac d2+\rs, \frac d2+1\} }\,
\end{equation}
where $2m \geq d+2$ depends on the number of iteration needed. 
In order to see that the decay of the advection term cannot be better that $(1+t)^{-\frac d2-1}$, notice that if $\alpha \geq 1 - \frac{d}{4}$, then the right-hand side of \eqref{here1} is bounded by a constant and the decay is dictated by \eqref{triangle} and \eqref{eq:propositionFinalResult} is attained directly.
  \end{proof}
\begin{remark}[About sharpness in the case $\rs=0$ and $\alpha>\frac 12$]
Notice that the combination of this upper bound with the lower bound in Theorem \ref{th1} shows that $\|\theta(t)\|_2\sim (1+t)^{-\frac d4-\frac{\rs}{2}}$ for $\rs=0$ and $\alpha>\frac 12$. 
\end{remark}
  Using the result in Proposition \ref{pr1}, we can now prove Lemma \ref{UB-diff}.
\begin{proof}[Proof of Lemma \ref{UB-diff}]
We can summarize the proof of this lemma in two steps:
\begin{description}

 \item[Step 1] Define $\eta:=\theta-T$ where $\theta$ satisfies \eqref{AD} and $T$ satisfies \eqref{HE}. Then, by the energy estimate applied to equation \eqref{eq:Diff}, we have 
\begin{multline}\label{eq:ode5}
\frac{d}{dt}((1+t)^{\beta}\|\eta(t)\|_2^2)\\\leq \beta (1+t)^{\beta-1}\int_{S(t)}|\hat{\eta}(\xi,t)|^2\, d\xi+2(1+t)^{\beta}\|\nabla T(t)\|_{\infty}\|u(t)\|_{2}\|\theta(t)\|_{2}\,,
\end{multline}
where the set $S(t)$ is defined in \eqref{def:S}.

\item[Step 2] Inserting in \eqref{eq:ode5} the result of Proposition \ref{pr1}, the assumption \eqref{new-assumption-u} and the estimate for the gradient of the heat equation
\begin{equation}\label{gradient-he}
\|\nabla T(t)\|_{\infty}\lesssim \|\theta_0\|_2\kappa^{-\frac d4-\frac 12}(1+t)^{-\frac d4-\frac 12}\,
\end{equation}
we deduce the existence of constants $C$ and $m$ depending on $d, \rs\|\theta_0\|_2$ and $\alpha$ such that \eqref{ubound-eta} holds.
\end{description}

\noindent\textbf{Proof of Step 1:} Testing \eqref{eq:Diff} by $\eta$ and integrating by parts we find
\begin{eqnarray*}
 \frac 12\frac{d}{dt}\|\eta(t)\|_2^2+\kappa\|\nabla \eta(t)\|_2^2&=&-\int (u\cdot\nabla\theta)(\theta-T)\\
 &=&-\int u\theta\cdot\nabla  T \\
 &\leq&\|\nabla T(t)\|_{\infty}\|u(t)\|_{2}\|\theta(t)\|_{2} \, ,
\end{eqnarray*}
where in the second-to-last estimate we used the incompressibility condition  for $u$ and in the last estimate we applied H\"older's inequality.
We apply Plancherel's theorem to the left-hand side of the previous equation 
$$\frac{d}{dt}\|\hat{\eta}(t)\|_2^2\leq-2\kappa \int_{\R^d}|\xi|^2|\hat{\eta}(\xi,t)|^2\, d\xi+2\|\nabla T(t)\|_{\infty}\|u(t)\|_{2}\|\theta(t)\|_{2}$$
and then consider the time-dependent decomposition of the space domain, i.e. $$\R^d=S(t)\cup S^c(t)$$ where~\footnote{as in the proof of Proposition \ref{pr1}, see \eqref{eq:S}.}
\begin{equation*}
S(t):=\left\{\xi\in \R^d\; |\;  |\xi|\leq \left(\frac{\beta}{2\kappa(1+t)}\right)^{\frac 12}\right\}\,.
\end{equation*}
Imposing the decomposition and using the non-negativity of the integrals we have 
\begin{align*}
 \frac{d}{dt}\|\hat{\eta}(t)\|_2^2
&=-2\kappa\int_{S(t)}|\xi|^2|\hat{\eta}(\xi,t)|^2\, d\xi \\
& \hspace{1cm} -2\kappa\int_{S^c(t)}|\xi|^2|\hat{\eta}(\xi,t)|^2\, d\xi+2\|\nabla T(t)\|_{\infty}\|u(t)\|_{2}\|\theta(t)\|_{2}\\
&\leq-2\kappa\int_{S^c(t)}|\xi|^2|\hat{\eta}(\xi,t)|^2\, d\xi+2\|\nabla T(t)\|_{\infty}\|u(t)\|_{2}\|\theta(t)\|_{2}\\
&\leq-2\kappa \frac{\beta}{2\kappa(1+t)} \int_{S^c(t)}|\hat{\eta}(\xi,t)|^2\, d\xi+2\|\nabla T(t)\|_{\infty}\|u(t)\|_{2}\|\theta(t)\|_{2}\\
&= -\frac{\beta}{(1+t)}\int_{\R^d}|\hat{\eta}(\xi,t)|^2\, d\xi \\
& \hspace{1cm} + \frac{\beta}{(1+t)}\int_{S(t)}|\hat{\eta}(\xi,t)|^2\, d\xi+2\|\nabla T(t)\|_{\infty}\|u(t)\|_{2}\|\theta(t)\|_{2}\,.
\end{align*}
Hence, since the first integral on the r.h.s. is positive, we obtain \eqref{eq:ode5}.

\noindent\textbf{Proof of Step 2:} Consider equation \eqref{eq:Diff} in Fourier space 
\begin{equation*}
 \partial_t\hat\eta+\kappa|\xi|^2\hat\eta=-\widehat{u\cdot \nabla\theta}\,,
\end{equation*}
the solution of which has the following representation
\begin{equation*}
 \hat \eta(\xi,t)=\int_0^te^{-\kappa |\xi|^2(t-s)}(-\widehat{u\cdot \nabla \theta})(\xi,s)\, ds\,.
\end{equation*}
where we used that $\hat{\eta}_0(\xi)=0$. Imitating the argument in Step 2 of Proposition \ref{pr1} we have 
\begin{equation}\label{rep}
|\hat\eta(\xi,t)|^2
\leq \left(\int_0^t e^{-\kappa|\xi|^2 (t-s)}|\widehat{u\cdot \nabla \theta}|(\xi,s)\, ds\right)^2\,
\end{equation}
and 
\begin{eqnarray*}
   |\widehat{u\cdot\nabla \theta}|
   &\leq&|\xi|\|\theta(t)\|_2\|u(t)\|_2\,.
  \end{eqnarray*}
Employing the bound \eqref{result:ub} in Proposition \ref{pr1} 
together with the assumption 
\begin{equation}\label{new-assumption-u-for-th}
\|u(t)\|_2\sim (1+t)^{-\alpha} \quad \mbox{with} \quad \alpha > \frac{1}{2} - \frac{d}{4}
\end{equation}
 we have, for $\rs \leq 1$
\begin{eqnarray*}
   |\widehat{u\cdot\nabla \theta}(t)|
   &\leq&|\xi|\|\theta(t)\|_2\|u(t)\|_2\leq C\kappa^{-m}|\xi|(1+t)^{-\frac d4-\frac{\rs}{2}- \alpha }\,.
  \end{eqnarray*}

Inserting this bound in \eqref{rep}, we obtain
\begin{eqnarray*}
 |\hat\eta(\xi,t)|^2
&\lesssim& C^2\kappa^{-2m}|\xi|^2\left(\int_0^t  e^{-\kappa|\xi|^2 (t-s)}(1+s)^{-\frac d4-\frac{\rs}{2}- \alpha }\, ds\right)^2\,\\
&\leq&  C^2\kappa^{-2m}|\xi|^2\left(\int_0^t (1+s)^{-\frac d4-\frac{\rs}{2}- \alpha }\, ds\right)^2\\
&\leq&  C^2\kappa^{-2m}\frac{1}{(-\frac d4-\frac{\rs}{2}- \alpha +1)^2} |\xi|^2\left( (1+t)^{-\frac d4-\frac{\rs}{2}- \alpha + 1}-1\right)^2\,,\\
\end{eqnarray*}
where we used $\left|e^{-\kappa|\xi|^2(t-s)}\right|\leq 1$.
Observe that $\left( (1+t)^{-\frac d4-\frac{\rs}{2}- \alpha + 1}-1\right)^2\lesssim 1$ if $\alpha >1 - \frac{d}{4} - \frac{\rs}{2}$.
Then
$$ |\hat\eta(\xi,t)|^2\lesssim C^2\kappa^{-2m}\frac{1}{(-\frac d4-\frac{\rs}{2}- \alpha +1)^2} |\xi|^2\,.$$
Integrating over $S(t)$ and using \eqref{triangle} we have
\begin{eqnarray*}
 \int_{S(t)}|\hat\eta(\xi,t)|^2\, d\xi
&\lesssim&  \frac{C^2\kappa^{-2m}}{(2+d)}\frac{1}{(-\frac d4-\frac{\rs}{2}- \alpha +1)^2} \left(\frac{\beta}{2\kappa(1+t)}\right)^{\frac{d+2}{2}}\\
&=& \frac{C^2 \beta^{\frac{d+2}{2}}}{(2+d)} \frac{1}{(-\frac d4-\frac{\rs}{2}- \alpha +1)^2}\kappa^{-2m-\frac d2-1}(1+t)^{-\frac d2-1}\,.
\end{eqnarray*}
Combining in \eqref{eq:ode5} estimate \eqref{gradient-he}
together with the the upper bound \eqref{result:ub}, we deduce
\begin{align*}
\frac{d}{dt}((1+t)^{\beta}\|\eta(t)\|_2^2)
\lesssim & \frac{C^2 \beta\beta^{\frac{d+2}{2}}}{(2+d)} \frac{1}{(-\frac d4-\frac{\rs}{2}- \alpha +1)^2}\kappa^{-2m-\frac d2-1}(1+t)^{-\frac d2-1+\beta-1}\\
& + C\|\theta_0\|_2\kappa^{-m-\frac d4-\frac 12} (1+t)^{-\frac d2-\frac{\rs}{2}- \alpha - \frac{1}{2} +\beta}\,.
\end{align*}
Integrating in time, choosing $\beta>\max\{\frac d2+1, \frac d2 + \frac{\rs}{2} + \alpha + \frac{1}{2} \}$, and dividing by $(1+t)^{\beta}$
\begin{align*}
\|\eta(t)\|_2^2
\lesssim & \frac{C^2 \beta\beta^{\frac{d+2}{2}}}{(2+d)} \frac{1}{(-\frac d4-\frac{\rs}{2}- \alpha +1)^2(-\frac d2-1+\beta)}\kappa^{-2m-\frac d2-1}(1+t)^{-\frac d2-1}\\
& + C\|\theta_0\|_2\kappa^{-m-\frac d4-\frac 12} \frac{1}{(-\frac d2-\frac{\rs}{2}- \alpha + \frac{1}{2} +\beta)}(1+t)^{-\frac d2-\frac{\rs}{2}- \alpha + \frac{1}{2}}\,.
\end{align*}
So we obtain
\begin{equation*}
\|\eta(t)\|_2^2\leq C
\kappa^{-m-\frac d4-\frac 12}(1+t)^{-\min\{\frac d2+1, \frac d2 + \frac{\rs}{2} + \alpha - \frac{1}{2} \}}
\end{equation*}
where $C$ is a constant depending on $d,\rs,\alpha$.
If instead $\rs\geq 1$, following the previous argument\footnote{ to easily see this, set $\rs=1$ in the computations above.}, choosing $\alpha > \frac{1}{2} - \frac{d}{4}$ we find 
\begin{equation*}
\|\eta(t)\|_2^2\leq C
\kappa^{-m-\frac d4-\frac 12}(1+t)^{-\min\{\frac d2+1, \frac d2 + \alpha \}}\,,
\end{equation*}
where $C$ is (another) constant depending on $d,\rs,\alpha$.

\end{proof}
\section{Theorem \ref{th2}}
The proof of Theorem \ref{th2} is based on the result in Theorem \ref{th1} and on the upper bound on the gradient of the solution of \eqref{AD}.
\begin{lemma}[Upper bound on the gradient]\label{UBG}
 Let the  assumptions in Proposition \ref{pr1} be satisfied.
 \begin{itemize}
 \item If 
 \begin{equation}\label{assumption-gradient2}
  \|\nabla u(t)\|_{\infty}\sim \frac{1}{(1+t)^{\nu}} \quad \mbox{ with } \; \nu>1\,,
 \end{equation}
 then there exists a constant $C>0$ depending on  $d$, $\rs$, $\|\nabla\theta_0\|_{2}, \alpha$ and $\nu$  such that 
\begin{equation*}
\|\nabla \theta(t)\|_2\leq C' \kappa^{-n-\frac 12}(1+t)^{-\min\{\frac d4+\frac{\rs}{2}+\frac 12, \frac d4+1\}}e^{\frac{[(1+t)^{-\nu+1}-1]}{-\nu+1}}\,.
\end{equation*}
\item If
 \begin{equation}\label{assumption-gradient1}
  \|\nabla u(t)\|_{\infty}\sim \frac{1}{(1+t)}\,,
 \end{equation}
 then there exists a constant $C''>0$ depending on $d$, $\rs$, $\|\nabla\theta_0\|_{2}, \alpha$ and $\nu$  such that 
 \begin{equation*}
  \|\nabla \theta(t)\|_2\leq C''\kappa^{-n-\frac12}(1+t)^{-\min\{\frac d4+\frac{\rs}{2}+\frac 12,\frac d4+1\}}\,.
 \end{equation*}
\item If 
 \begin{equation}\label{assumption-gradient3}
  \|\nabla u(t)\|_{\infty}\sim \frac{1}{(1+t)^{\nu}} \quad \mbox{ with } \; 0\leq\nu<1\,,
 \end{equation}
 then there exists a constant $C'''>0$ depending on $d$, $\rs$, $\|\nabla\theta_0\|_{2}, \alpha$ and $\nu$  such that 
 \begin{equation*}
\|\nabla \theta(t)\|_2\leq C''' \kappa^{-n-\frac 12}(1+t)^{-\min\{\frac d4+\frac{\rs}{2}+\frac 32,\frac d4+2\}}e^{\frac{[(1+t)^{-\nu+1}-1]}{-\nu+1}}\,.
\end{equation*}
\end{itemize}
In the statements above $n=\max\{\frac d4+\frac{\rs}{2},m\}$. 
\end{lemma}
\subsection{Proof of Theorem  \ref{th2}}\label{proof-th2}
\begin{proof}[Proof of Theorem \ref{th2}]
The interpolation inequality 
$$\|\theta (t)\|_2^2\leq\|\nabla\theta(t)\|_2\|\nabla^{-1}\theta(t)\|_2\,,$$
together with the lower bound in Theorem \ref{th1} (holding for $\rs<1$) and the upper bound in Lemma \ref{UBG} yield
\begin{equation*}
\begin{array}{rlll}
  \|\nabla^{-1}\theta(t)\|_2&\gtrsim&C \kappa^{\frac d2-\rs+m+\frac 12}e^{-\frac{[(1+t)^{-\nu+1}-1]}{-\nu+1}}(1+t)^{-\frac d4-\frac{\rs}{2}+\frac 12} & \mbox{ with } \eqref{assumption-gradient2} \\
  \\
  \|\nabla^{-1}\theta(t)\|_2&\geq& C \kappa^{\frac d2-\rs+m+\frac 12}(1+t)^{-\frac d4-\frac{\rs}{2}+\frac 12}  & \mbox{ with } \eqref{assumption-gradient1}\\
  \\
    \|\nabla^{-1}\theta(t)\|_2&\geq& C \kappa^{\frac d2-\rs+m+\frac 12}e^{-\frac{[(1+t)^{-\nu+1}-1]}{-\nu+1}}(1+t)^{-\frac d4-\frac{\rs}{2}+\frac 32}   & \mbox{ with } \eqref{assumption-gradient3}\,.
\end{array}
\end{equation*}
\end{proof}
\noindent Notice that our lower bound becomes trivial when $\kappa\rightarrow 0$.
\begin{proof}[Proof of Corollary \ref{cor}]
The statement can be easily obtained by combining Theorem \ref{th1} and Lemma \ref{UBG} in the interpolation inequality \eqref{interpol}.
\end{proof}
\subsection{Proof of the lemma}
Inspired by  \cite{SchoScho2006}, where bounds on the derivative of QG equation are obtained, the proof of Lemma \ref{UBG} results from the combination of standard energy estimates, the upper bound in Proposition \ref{pr1} and a classical Gronwall-type argument \cite{lakshmikantham2015stability}.
\begin{proof}[Proof of Lemma \ref{UBG}]
We start by testing the advection-diffusion equation \eqref{AD} with $\Delta\theta$
\begin{eqnarray*}
 \frac 12\frac{d}{dt}\int |\nabla \theta|^2\, dx+\kappa \int |\Delta\theta|^2\, dx
 &=& -\int\partial_ju_i\partial_i\theta\partial_j\theta\, dx\\
 &\leq& \left|-\int\partial_ju_i\partial_i\theta\partial_j\theta\, dx\right|\\
 &\leq&\|\nabla u\|_{\infty}\|\nabla \theta\|_2^2\,,
\end{eqnarray*}
thus
\begin{equation}\label{test}
 \frac{d}{dt} \|\nabla \theta\|_{2}^2+2\kappa \|\Delta\theta\|_2^2
\leq
2 \|\nabla u\|_{\infty}\|\nabla \theta\|_2^2\,.
\end{equation}
Define 
\begin{equation*}
W(t)=\left\{\xi\in \R^d: |\xi|\leq \left(\frac{\mu}{\kappa(1+t)}\right)^{\frac 12}\right\}\,
\end{equation*}
and apply Plancherel's identity to get
\begin{eqnarray*}
 \|\Delta\theta(t)\|_2^2&=&\int |\xi|^4|\hat\theta(\xi,t)|^2\, d\xi\\
                  &\geq& \int_{W^c(t)}|\xi|^4|\hat\theta(\xi,t)|^2\, d\xi\\
                  &\geq&\frac{\mu}{\kappa(1+t)}\int_{W^c(t)}|\xi|^2|\hat\theta(\xi,t)|^2\, d\xi\\
                  &=&\frac{\mu}{\kappa(1+t)}\left(\int_{\R^d}|\xi|^2|\hat\theta(\xi,t)|^2\,d\xi-\int_{W(t)}|\xi|^2|\hat{\theta}(\xi,t)|^2\, d\xi\right)\\
                  &\geq&\frac{\mu}{\kappa(1+t)}\left(\|\nabla \theta(t)\|_2^2-\frac{\mu}{\kappa(1+t)}\int_{\R^d}|\hat{\theta}(\xi,t)|^2\, d\xi\right)\\
                  &=&\frac{\mu}{\kappa(1+t)}\left(\|\nabla \theta(t)\|_2^2-\frac{\mu}{\kappa(1+t)}\|\theta(t)\|_2^2\right)\,.
\end{eqnarray*}
Plug the result of Proposition \ref{pr1}, i.e. the upper bound~\footnote{Recall that the upper bound  holds under the assumption
$$\|u(t)\|_{2}\sim (1+t)^{-\alpha}\, \mbox{ with } \alpha > \frac 12-\frac d4\,.$$ }
$$
   \|\theta(t)\|_2\lesssim C \kappa^{-n}(1+t)^{ -\min\{\frac d4+\frac{\rs}{2}, \frac d4+\frac 12\} }
  $$
  where $n=\max\{\frac d4+\frac{\rs}{2}, m\}$
in the previous estimate to obtain
\begin{eqnarray*}
 \|\Delta\theta(t)\|_2^2
                  &\geq &\frac{\mu}{\kappa(1+t)}\left(\|\nabla \theta(t)\|_2^2
                  -\frac{\mu C^2}{\kappa(1+t)}\kappa^{-2n}(1+t)^{-\min\{\frac d2+\rs,\frac d2+1\}}\right)\\
                  &=&\frac{\mu}{\kappa(1+t)}\left(\|\nabla \theta(t)\|_2^2
                  -\mu C^2 \kappa^{-2n-1}(1+t)^{-\min\{\frac d2+\rs+1,\frac d2+2\}}\right)\,.
\end{eqnarray*}
Inserting this lower bound in \eqref{test} we obtain the differential inequality
\begin{multline}\label{diff-ineq}
 \frac{d}{dt} \|\nabla \theta(t)\|_2^2+\frac{2\mu}{(1+t)}\|\nabla \theta(t)\|_2^2 \\
 \leq 2\|\nabla u(t)\|_{\infty}\|\nabla \theta(t)\|_2^2
 +2\mu^2C^2\kappa^{-m-2}(1+t)^{-\min\{\frac d2+\rs+2,\frac d2+3\}}\,.
\end{multline}
Set $X(t)=\|\nabla \theta(t)\|_2^2$ and rewrite \eqref{diff-ineq} as
\begin{equation}\label{diff-ineq3}
 \frac{d}{dt} X(t)\leq a(t)X(t)+b(t)\,,
\end{equation}
where $$a(t)=-\frac{2\mu}{(1+t)}+2\|\nabla u(t)\|_{\infty} \quad \mbox{ and }\quad  
b(t)=2\mu^2C^2\kappa^{-2n-1}(1+t)^{-\min\{\frac d2+\rs+2,\frac d2+3\}}\,.$$
Define $q(t)=X(t)e^{-\int_{0}^{t}a(s)\, ds}$ with $q(0)=\|\nabla \theta_0\|_2^2$. Then 
$$q'(t)=[X'(t)-a(t)X(t)]e^{-\int_{0}^{t}a(s)\, ds}\leq b(t)e^{-\int_{0}^{t}a(s)\, ds}\,,$$
and therefore
$$q(t)\leq \|\nabla \theta_0\|_2^2+\int_0^t b(s) e^{-\int_{0}^{s} a(\tau)\, d\tau}\, ds\,,$$
which, by the definition of $q$, turns into
\begin{equation}\label{solution}
X(t)\leq \|\nabla \theta_0\|_2^2e^{\int_{0}^{t}a(s)\, ds}+\left(\int_0^t b(s) e^{-\int_{0}^{s}a(\tau)\, d\tau}\, ds\right) e^{\int_{0}^{t}a(s)\, ds}\,.
\end{equation}
We now split the analysis in three cases:
\begin{itemize}
\item[1)]
 Assume \eqref{assumption-gradient2}. Then \eqref{diff-ineq3} holds with 
 $$a(t)=-\frac{2\mu}{(1+t)}+ \frac{2}{(1+t)^{\nu}} \quad \mbox{ and }\quad 
  b(t)=2\mu^2C^2\kappa^{-2n-1}(1+t)^{-\min\{\frac d2+\rs+2,\frac d2+3\}}\,.$$
The conclusion is obtained by computing the right-hand side of \eqref{solution}:
the term multiplying $\|\nabla \theta_0\|_2^2$ is
$$e^{\int_{0}^{t}a(s)ds}=(1+t)^{-2\mu}e^{2\frac{[(1+t)^{-\nu+1}-1]}{-\nu+1}}$$
and, using that for all $s>0$ and $\nu>1$ we have $e^{-2\frac{[(1+s)^{-\nu+1}-1]}{-\nu+1}}\leq 1$ and choosing $2\mu>\min\{ \frac d2+\rs+1, \frac d2+2\}$, the second term can be bounded as follows
\begin{multline*}
\left(\int_0^tb(s)e^{-\int_{0}^{s}a(\tau)\, d\tau}\, ds\right)\; e^{\int_{0}^{t}a(s)\, ds} \\
\leq2\mu^2C^2\kappa^{-2n-1}\frac{(1+t)^{-\min\{\frac d2+\rs+1, \frac d2+2\}}}{-\min\{\frac d2+\rs+1-2\mu, \frac d2+2-2\mu\}} e^{2\frac{[(1+t)^{-\nu+1}-1]}{-\nu+1}}\,.
\end{multline*}
Hence we obtain
\begin{multline*}
	\|\nabla \theta(t)\|_2^2\leq \|\nabla \theta_0\|_2^2(1+t)^{-2\mu}e^{2\frac{[(1+t)^{-\nu+1}-1]}{-\nu+1}}\\
	+2\mu^2C^2\kappa^{-2n-1}\frac{(1+t)^{-\min\{\frac d2+\rs+1, \frac d2+2\}}}{-\min\{\frac d2+\rs+1-2\mu, \frac d2+2-2\mu\}} e^{2\frac{[(1+t)^{-\nu+1}-1]}{-\nu+1}}\,.
\end{multline*}
Because of our choice of $\mu$ we conclude 
$$\|\nabla \theta(t)\|_2^2\leq (C')^2 \kappa^{-2n-1}(1+t)^{-\min\{\frac d2+\rs+1, \frac d2+2\}}e^{2\frac{[(1+t)^{-\nu+1}-1]}{-\nu+1}}\,,$$
where 
$$C':=\left(\|\nabla\theta_0\|_2^2 \kappa^{2n+ 1}+\frac{2\mu^2C^2}{-\min\{\frac d2+\rs+1-2\mu, \frac d2+2-2\mu\}} \right)^{\frac 12}\,.$$
\item[2)] Assume \eqref{assumption-gradient1}. In this case the coefficients in \eqref{solution} are
$$a(t)=\frac{-2\mu+2}{(1+t)} \quad\mbox{ and }\quad b(t)=2\mu^2C^2\kappa^{-2n-1}(1+t)^{-\min\{\frac d2+\rs+2,\frac d2+3\}}\,.$$
We compute
$$e^{\int_{0}^{t}a(s)\, ds}=(1+t)^{-2\mu+2}$$
and, choosing $2\mu>\min\{\frac d2+\rs+1, \frac d2+2\}$, we have
\begin{multline*}
\left(\int_0^tb(s)e^{-\int_{0}^{t}a(\tau)\, d\tau}\, ds\right)\; e^{\int_{0}^{t}a(\tau)\, d\tau}\\
\leq\frac{2\mu^2C^2\kappa^{-2n-1}}{-\min\{\frac d2+\rs+1-2\mu, \frac d2+2-2\mu\}}(1+t)^{-\min\{\frac d2+\rs+1, \frac d2+2\}    }\,.
\end{multline*}
Substituting in \eqref{solution}, we obtain
\begin{multline*}
\|\nabla \theta(t)\|_2^2 \leq \|\nabla \theta_0\|_2^2(1+t)^{-2(\mu-1)}\\
+\frac{2\mu^2C^2\kappa^{-2n-1}}{-\min\{\frac d2+\rs+1-2\mu, \frac d2+2-2\mu\}}(1+t)^{-\min\{\frac d2+\rs+1, \frac d2+2\}    }\,,
\end{multline*}
which, because of our choice of $\mu$, implies 
$$\|\nabla \theta(t)\|_2^2\leq (C'')^2\kappa^{-2n-1}(1+t)^{-\min\{\frac d2+\rs+1, \frac d2+2\}}\,,$$
where 
$$C'':=\left(\|\nabla \theta_0\|_2^2 \kappa^{2n+ 1} +\frac{2\mu^2C^2}{-\min\{\frac d2+\rs+1-2\mu, \frac d2+2-2\mu\}}\right)^{\frac 12}\,.$$

\item[3)] Finally assume \eqref{assumption-gradient3}, then
 $$a(t)=-\frac{2\mu}{(1+t)}+ \frac{2}{(1+t)^{\nu}} \quad \mbox{ and }\quad  b(t)=2\mu^2C^2\kappa^{-2n-1}(1+t)^{-\min\{\frac d2+\rs+2,\frac d2+3\}}\,.$$
The fundamental solution can be computed easily
$$e^{\int_{0}^{t}a(s)ds}=(1+t)^{-2\mu}e^{2\frac{[(1+t)^{-\nu+1}-1]}{-\nu+1}}\,.$$
For the second term of the right-hand side of \eqref{solution} we first compute
\begin{multline*}
\int_0^tb(s)e^{-\int_{0}^{s}a(\tau)\, d\tau}\, ds \\
=2\mu^2C^2\kappa^{-2n-1}\int_0^t (1+s)^{-\min\{\frac d2+\rs+2,\frac d2+3\}}
(1+s)^{2\mu}e^{-2\frac{[(1+s)^{-\nu+1}-1]}{-\nu+1}}\, ds\,,
\end{multline*}
and notice that for $\nu\in[0,1)$ and all $s\geq 0$
\footnote{It is easy to see this by Taylor expansion about $\nu=1$. To leading order
$$e^{-2\frac{[(1+s)^{\varepsilon}-1]}{\varepsilon}}\sim e^{-\frac{2}{\varepsilon}(\varepsilon \log(1+s))} $$
where $\varepsilon:=1-\nu\ll 1$.}
$$e^{-2\frac{[(1+s)^{-\nu+1}-1]}{-\nu+1}}\leq (1+s)^{-2}$$
and therefore, choosing $\mu>\min\{\frac d2+\rs+3,\frac d2+4 \}$ we have 
\begin{multline*}
\left(\int_0^tb(s)e^{-\int_{0}^{s}a(\tau)\, d\tau}\, ds\right)\; e^{\int_{0}^{t}a(s)\, ds} \\
\leq2\mu^2C^2\kappa^{-2n-1}\frac{(1+t)^{-\min\{\frac d2+\rs+3,\frac d2+4\}}}{-\min \{\frac d2+\rs+3-2\mu, \frac d2+4-2\mu\}} e^{2\frac{[(1+t)^{-\nu+1}-1]}{-\nu+1}}\,.
\end{multline*}
In conclusion, we obtain
\begin{multline*}
	\|\nabla \theta(t)\|_2^2\leq \|\nabla \theta_0\|_2^2(1+t)^{-2\mu}e^{2\frac{[(1+t)^{-\nu+1}-1]}{-\nu+1}}\\
	+2\mu^2C^2\kappa^{-2n-1}\frac{(1+t)^{-\min\{\frac d2+\rs+3,\frac d2+4\}}}{-\min \{\frac d2+\rs+3-2\mu, \frac d2+4-2\mu\}} e^{2\frac{[(1+t)^{-\nu+1}-1]}{-\nu+1}}\,,
\end{multline*}
which, by our choice of $\mu$, implies
$$\|\nabla \theta(t)\|_2^2\leq (C''')^2 \kappa^{-2n-1}(1+t)^{-\min\{\frac d2+\rs+3,\frac d2+4\}}e^{2\frac{[(1+t)^{-\nu+1}-1]}{-\nu+1}}\,,$$
where 
$$C''':=\left(\|\nabla\theta_0\|_2^2 \kappa^{2n + 1} +\frac{2\mu^2C_2^2}{-\min \{\frac d2+\rs+3-2\mu, \frac d2+4-2\mu\}} \right)^{\frac 12}\,.$$
\end{itemize}
\end{proof}
\section{Conclusion}

In the present paper, we derive lower bounds for quantities characterizing the effectiveness of mixing in passive scalar transport.
For the discussed cases, the initial data is specified by so-called decay characters $\rs$ \cite{niche2015decay, niche2016decay} and the divergence-free vector field $u$ is constrained in its temporal decay.
With this, our approach follows the reasoning of \cite{DM2018}, i.e.\ proving lower bounds on $\|\theta\|_2$ and $\| \nabla^{-1} \theta\|_2$ accordingly imply bounds on the filamentation length $\lambda(t)$.
Specifically, employing the Fourier splitting method \cite{Scho1986, Scho1991} and under the assumptions
\begin{itemize}
\item $\theta_0\in L^2(\R^d)$ with decay character $\rs$ such that $0<\rs<1$ for $d=2$ and $-\frac 12<\rs<1$ for $d=3$,\\
\item $\|u(t)\|_{2}\sim (1+t)^{-\alpha}$ with $\alpha>\frac \rs 2+\frac 12$,\\
\item $\|\nabla u(t)\|_{\infty}\sim (1+t)^{-\nu}$,
\end{itemize}
we prove the following lower bound for the filamentation length $\lambda$ (defined in \eqref{def-lambda}): for $\nu>1$
 $$ \lambda(t)\gtrsim C  \kappa^{-\frac d4-\frac{\rs}{2}+m+\frac 12}e^{-\frac{[(1+t)^{-\nu+1}-1]}{-\nu+1}}(1+t)^{\frac 12} \rightarrow \infty \quad \mbox{ as } t\rightarrow \infty \,,$$
 for $\nu=1$
 $$
  \lambda(t)\gtrsim C \kappa^{-\frac d4-\frac{\rs}{2}+m+\frac 12}(1+t)^{\frac 12} \rightarrow \infty \quad \mbox{ as } t\rightarrow \infty \,,$$
  and for $0\leq \nu <1$
  $$ \lambda(t)\gtrsim   C\kappa^{-\frac d4-\frac{\rs}{2}+m+\frac 12}e^{-\frac{[(1+t)^{-\nu+1}-1]}{-\nu+1}}(1+t)^{\frac 32}\rightarrow 0 \quad \mbox{ as } t\rightarrow \infty \,.$$
  This result is contained in Corollary \ref{cor} and its proof is based on the combination of Theorem \ref{th1} and Theorem \ref{th2}.

Notice that, according to Remark \ref{initial-data}, a class of initial data for which $\rs\in (-\frac 12,1)$ in $\R^3$ is given by $\theta_0\in L^2(\R^3)\cap L^p(\R^3)$ with $1<p<\frac 65$. In $\R^2$, instead, an example of initial data such that $\rs\in (0,1)$ is given by $\theta_0\in L^2(\R^2)$ such that $|\hat{\theta_0}(\xi)|\sim |\xi|^{\alpha}$ for $|\xi|\leq \delta$ and $0<\alpha<1$. 

We want to conclude by comparing the behavior of $\lambda$ expressed in \eqref{lambda-lb-ad} when considering either $u\equiv0$ (pure diffusion) or $\kappa=0$ (pure advection). 
Let us first consider the pure diffusion equation 
$$\partial_tT-\kappa\Delta T=0\,, \quad T(x,0)=\theta_0(x)\quad \mbox{ in } \R^3\,,$$
under the assumptions $\rs(\theta_0)=0$, for simplicity.
The lower bound for $T$
$$\|T (t)\|_{2}\gtrsim_{M,\delta,\kappa} (1+t)^{-\frac 34}$$ 
(see Lemma \ref{LB-HE} with $\rs=0$) together with the upper bound
$$\|\nabla T(t)\|_2\lesssim_{\kappa}(1+t)^{-\frac 34-\frac 12}$$
yields
\begin{equation}\label{lambda-lb-pd}
\lambda(t)\gtrsim_{M,\delta,\kappa}(1+t)^{\frac 12}\,.
\end{equation}
On the other hand, considering the pure advection equation 
$$\partial_t \theta+u\cdot \nabla \theta=0, \quad \theta(x,0)=\theta_0(x),$$
imitating the computations in \cite{LTD2011} and using that the energy is exactly conserved, i.e. $\|\theta(t)\|_2=\|\theta_0\|_2$
for all time, we obtain 
\begin{equation}\label{lambda-lb-pa}
 \begin{array}{llll}
  \lambda(t)\gtrsim e^{-\frac{[(1+t)^{-\nu+1}-1]}{-\nu+1}}\quad &\mbox{ under assumption } \eqref{gradient-condition1} \,,\\
  \lambda(t)\gtrsim (1+t)^{-1}\quad &\mbox{ under assumption } \eqref{gradient-condition2}\,,\\
  \lambda(t)\gtrsim e^{-\frac{[(1+t)^{-\nu+1}-1]}{-\nu+1}}\quad &\mbox{ under assumption } \eqref{gradient-condition3}\,.\\
 \end{array}
\end{equation}
If we denote the lower bound estimates for $\lambda$ obtained in \eqref{lambda-lb-ad}, \eqref{lambda-lb-pd} and \eqref{lambda-lb-pa} generically with $g(t)$, then we can described the behavior of the filamentation length (under the assumptions $\rs(\theta_0)=0$, $\|\theta_0\|_{2}<\infty$ and $\|\nabla^{-1}\theta_0\|_{2}<\infty)$  as
$$\lambda(t)\gtrsim g(t)\stackrel{t\rightarrow \infty}{\longrightarrow} g^{\infty}$$
with $g^{\infty}$ specified in the following chart:
\vspace{0.5cm}
\begin{center}
\begin{tabular}{ |c | c | c | c| } 
\hline
& $0 \leq \nu < 1$ & $\nu = 1$ & $\nu > 1$ \\
\hline
Pure Advection & $g^{\infty}= 0$ & $g^{\infty}= 0$ & $g^{\infty}= e^{-\frac{1}{\nu-1}}$ \\
\hline 
Advection-Diffusion & $g^{\infty}=0$ & $g^{\infty}=\infty$ & $g^{\infty}=\infty$ \\
\hline 
Pure Diffusion &$g^{\infty}=\infty$ & $g^{\infty}=\infty$   & $g^{\infty}=\infty$ \\
\hline
\end{tabular}
\end{center}
\vspace{0.5cm}

It would be interesting to investigate whether the lower bound estimates for $\lambda(t)$ are sharp for some specific flow and to see whether the class of ``admissible'' velocity fields we can consider in our analysis could be further extended by using a combination of our argument with the Aronson-type estimate of Maekawa \cite{Maek08}.
Another interesting question is whether, going beyond a perturbative analysis, enhanced dissipation phenomena can be observed also in $\R^d$.
In fact all the enhanced dissipation results available so far have been derived in periodic or bounded domains (at least in one direction).
Furthermore, we observe that the enhanced dissipation phenomena in $\R^2$ observed in \cite{CZD20} was proven by choosing a velocity field growing at infinity (thus not in any $L^p$ space and, in particular, not in the class of velocity fields we can consider in our analysis).
Finally, we want to mention the recent work of Bedrossian, Blumenthal and Punshon-Smith \cite{BBP21}, where uniform upper bounds of the type $\|\theta\|_{H^{-1}}\leq D_{\kappa}(\omega, u)\|\theta\|_{H^{1}}$ were obtained in $\mathbb{T}^d$ and $D_{\kappa}(\omega,u)$ is a P-a.s. finite random constant.
In Remark 1.7 the authors argue that if their result would be proven to be sharp, then their results would imply the Batchelor-scale conjecture as formulated by Charles Doering and others in \cite{DM2018}. 
We refer the reader also to \cite{Pappa21} for a new interesting stochastic approach to the question of upper bounds for the $H^{-1}$-norm of the passive tracer advected by an Ornstein-Uhlenbeck velocity field.
We remark that the Batchelor-scale conjecture, has been observed in numerical simulations (for example \cite{DM2018,CGGA2017}), reduced dyadic models \cite{DM2018-2}, and it is believed to hold on bounded domains or the torus.
What role the geometry plays in the Batchelor-scale conjecture, remains to be explained.
\appendix
\section{Pure Advection}
Consider the pure advection equation in the whole space:
\begin{equation}\label{PA}
\begin{cases}
 \partial_t\theta+u\cdot \nabla \theta=0 & \mbox{ in } \R^d\times (0,\infty)\\
 \nabla \cdot u=0 & \mbox{ in } \R^d\times (0,\infty)\\
 \theta(x,0)=\theta_0(x) & \mbox{ in } \R^d\,.
\end{cases}
\end{equation}
under one of the following conditions for the gradient of the velocity field:
\begin{align}
 \|\nabla u(t)\|_{\infty}&\sim (1+t)^{-\nu} \qquad \mbox{ with } \nu>1 \label{gc1}\\
 \|\nabla u(t)\|_{\infty}&\sim (1+t)^{-1} \label{gc2}\\
 \|\nabla u(t)\|_{\infty}&\sim (1+t)^{-\nu} \qquad \mbox{ with } 0\leq\nu<1 \label{gc3}
\end{align}
Imitating the argument in of Lin, Thiffeault and Doering \cite{LTD2011}, 
we compute  $\lambda$ (defined in \eqref{def-lambda}) for the pure advection equation on the whole space:
using the identity 
\begin{equation*}
\frac{d}{dt}\|\nabla^{-1}\theta(t)\|_2^2=2\int \nabla^{-1}\theta\,\cdot \nabla u\,\cdot \nabla^{-1}\theta\, dx\,,
\end{equation*}
which is derived by testing equation \eqref{PA} with $\phi=\Delta^{-1}\theta$, integrating by parts and noticing that
\begin{equation*}
\begin{array}{rlll}
\int u\cdot\nabla\theta\,\Delta^{-1}\theta\, dx&=& \int \nabla \cdot(u\theta)\Delta^{-1}\theta\\
&=&\int \nabla \cdot (u\, \Delta\phi)\, \phi\\
&=&-\int \sum_{i,j}\partial_i(u^i\partial_{jj}\phi)\phi\\
&=&-\int \sum_{i,j} u^i\partial_{jj}\phi\partial_i\phi\\
&=&\int \sum_{i,j}\partial_j(u^i\partial_i\phi)\partial_j\phi\\
&=&\int \sum_{i,j} \partial_ju^i|\partial_i\phi|^2+\int u^i\partial_i\frac{|\partial_j\phi|^2}{2}\, dx\\
&=&\int \nabla^{-1}\theta\,\cdot \nabla u\,\cdot \nabla^{-1}\theta\, dx\,,
\end{array}
\end{equation*}
we obtain 
$$\frac{d}{dt}\|\nabla^{-1}\theta(t)\|_2^2\geq-2\|\nabla u(t)\|_{\infty}\|\nabla^{-1}\theta(t)\|_2^2\,.$$
Inserting condition \eqref{gc1} or \eqref{gc3} and applying Gronwall's inequality, we have 
$$\|\nabla^{-1}\theta(t)\|_2^2\gtrsim \|\nabla^{-1}\theta_0\|_2^2e^{-2\frac{[(1+t)^{-\nu+1}-1]}{-\nu-1}}\,.$$
 Instead, inserting condition  \eqref{gradient-condition2} we get
$$\|\nabla^{-1}\theta(t)\|_2^2\gtrsim \|\nabla^{-1}\theta_0\|_2^2(1+t)^{-2}\,.$$
Since $\|\theta(t)\|_2=\|\theta_0\|_2$ for all $t\geq 0$, we have 
\begin{equation*}
 \begin{array}{llll}
  \lambda(t)\gtrsim C\, e^{-\frac{[(1+t)^{-\nu+1}-1]}{-\nu+1}}\quad &\mbox{ under assumption } \eqref{gc1} \,,\\
  \lambda(t)\gtrsim C\,(1+t)^{-1}\quad &\mbox{ under assumption } \eqref{gc2}\,,\\
  \lambda(t)\gtrsim C\,e^{-\frac{[(1+t)^{-\nu+1}-1]}{-\nu+1}}\quad &\mbox{ under assumption } \eqref{gc3}\,,\\
 \end{array}
\end{equation*}
where $C=\frac{\|\nabla^{-1}\theta_0\|_2}{\| \theta_0\|_2}$.
%

\end{document}